\documentclass[reqno]{amsart}

\usepackage{tabu}
\usepackage{amssymb}
\usepackage{mathtools}
\usepackage{a4wide,amsmath}
\usepackage{mathrsfs}
\usepackage{amsthm}
\numberwithin{equation}{section}
\numberwithin{figure}{section}
\numberwithin{table}{section}
\usepackage{bbm}
\usepackage{subfig}
\usepackage{enumerate}
\usepackage[section]{placeins}
\usepackage{graphicx}		  
\usepackage{ifpdf}
\ifpdf
\DeclareGraphicsExtensions{.pdf,.eps,.jpg,.png}	
\usepackage[suffix=]{epstopdf}
\fi
\usepackage{xcolor}
\usepackage[utf8]{inputenc}
\usepackage{hyperref}
\hypersetup{hidelinks}

\long\def\MSC#1\EndMSC{\def\arg{#1}\ifx\arg\empty\relax\else
	{\narrower\noindent%
		{2010 Mathematics Subject Classification}: #1\\} \fi}
\long\def\PACS#1\EndPACS{\def\arg{#1}\ifx\arg\empty\relax\else
	{\narrower\noindent%
		{PACS numbers}: #1}\fi}
\long\def\KEY#1\EndKEY{\def\arg{#1}\ifx\arg\empty\relax\else
	{\narrower\noindent%
		Keywords: #1\\}\fi}

\theoremstyle{plain}
\newtheorem{theorem}{Theorem}[section]
\newtheorem{lemma}[theorem]{Lemma}
\newtheorem{proposition}[theorem]{Proposition}
\newtheorem{corollary}[theorem]{Corollary}
\theoremstyle{definition}
\newtheorem{definition}[theorem]{Definition}

\newtheorem{assumption}[theorem]{Assumption}
\theoremstyle{remark}
\newtheorem{remark}[theorem]{Remark}

\newcommand{\norm}[1]{\lVert#1\rVert}
\newcommand{\abs}[1]{\lvert#1\rvert} 
\newcommand{\inner}[1]{\langle#1\rangle}

\newcommand{\essinf}{\mathop{\textup{ess\,inf}}}
\newcommand{\esssup}{\mathop{\textup{ess\,sup}}}

\newcommand{\rum}[1]{\mathbb{#1}}

\newcommand{\I}{\mathrm{i}}    
\newcommand{\e}{\mathrm{e}}    
\newcommand{\di}{\mathrm{d}}   

\newcommand{\DLambda}{D\!\Lambda}

\begin{document}

	\title[Monotonicity-based reconstruction of extreme inclusions in EIT]{Monotonicity-based reconstruction of extreme inclusions in electrical impedance tomography}
	
	\author[V.~Candiani]{Valentina Candiani}
	\address[V.~Candiani]{Department of Mathematics and Systems Analysis, Aalto University, P.O. Box~11100, 00076 Helsinki, Finland.}
	\email{valentina.candiani@aalto.fi}
	
	\author[J.~Dard\'e]{J\'er\'emi Dard\'e}
	\address[J.~Dard\'e]{Institut de Math\'ematiques de Toulouse, Universit\'e Paul Sabatier, 118 route de Narbonne, GF-31062 Toulouse Cedex 9, France.}
	\email{jeremi.darde@math.univ-toulouse.fr}
	
	\author[H.~Garde]{Henrik Garde}
	\address[H.~Garde]{Department of Mathematics, Aarhus University, Ny Munkegade 118, 8000 Aarhus C, Denmark.}
	\email{garde@math.au.dk}
	
	\author[N.~Hyv\"onen]{Nuutti Hyv\"onen}
	\address[N.~Hyv\"onen]{Department of Mathematics and Systems Analysis, Aalto University, P.O. Box~11100, 00076 Helsinki, Finland.}
	\email{nuutti.hyvonen@aalto.fi}
	
	\begin{abstract}
		The monotonicity-based approach has become one of the fundamental methods for reconstructing inclusions in the inverse problem of electrical impedance tomography. Thus far the method has not been proven to be able to handle extreme inclusions that correspond to some parts of the studied domain becoming either perfectly conducting or perfectly insulating. The main obstacle has arguably been establishing suitable monotonicity principles for the corresponding Neumann-to-Dirichlet boundary maps. In this work, we tackle this shortcoming by first giving a convergence result in the operator norm for the Neumann-to-Dirichlet map when the conductivity coefficient decays to zero and/or grows to infinity in some given parts of the domain. This allows passing the necessary monotonicity principles to the limiting case. Subsequently, we show how the monotonicity method generalizes to the definite case of reconstructing either perfectly conducting or perfectly insulating inclusions, as well as to the indefinite case where the perturbed conductivity can take any values between, and including, zero and infinity. 
	\end{abstract}	
	\maketitle
	
	\KEY
	electrical impedance tomography, 
	perfectly insulating,
	perfectly conducting,
	monotonicity method.
	\EndKEY
	
	\MSC
	35R30,
	35R05,
	47H05.
	\EndMSC
	
	\section{Introduction}
	{\em Electrical impedance tomography} (EIT) is an imaging technology for retrieving information about the interior electrical conductivity of a physical object through measurements of electric current and voltage at electrodes placed on its boundary. The corresponding mathematical inverse problem is known as Calder\'on's problem or the \emph{inverse conductivity problem}. In mathematical terms, the inverse conductivity problem consists of reconstructing a coefficient inside a divergence term of a second order elliptic partial differential equation, based on information about the {\em Neumann-to-Dirichlet} (ND) map that encodes all possible Cauchy data (pairs of current densities and voltages) at the domain boundary. For more information on the inverse conductivity problem, we refer to \cite{Borcea2002a,Borcea2002,Uhlmann2009,Calderon1980} and the references therein.
	
	In this paper we consider the related, easier, problem of \emph{inclusion detection} with the goal of retrieving information on an unknown conductivity perturbation in a known background. Typically, inclusion detection algorithms aim at reconstructing the (outer shape of the) support of the perturbation, often referred to as \emph{inclusions}. We distinguish between \emph{definite} inclusions that correspond to conductivity perturbations that are either only positive or only negative and \emph{indefinite} inclusions defined by perturbations that may take both positive and negative values. 
	
	In the standard setting of EIT, where the conductivity is bounded away from zero and infinity, there are several approaches to solving the inclusion detection problem. Among some of the popular techniques are the \emph{factorization method} \cite{Bruhl2001,Bruhl2000,Kirsch2008,Hanke2015,Harrach13b,Gebauer2007,Hyvonen2004,Lechleiter2008b,Lechleiter2006} and the \emph{enclosure method} \cite{Ikehata1999a,Ikehata2000c,Brander_2015}. A more recent approach, which we consider in this article, is the \emph{monotonicity method} \cite{Tamburrino2002,Harrach10,Harrach13} that relies on the monotonicity properties of the map $\sigma\mapsto \Lambda(\sigma)$, where $\sigma$ represents a conductivity coefficient and $\Lambda(\sigma)$ the corresponding ND map. The monotonicity method is flexible in the sense that it directly handles both definite and indefinite inclusions, even when only having access to a part of the object boundary, with non-homogeneous background conductivity, with inclusions having several connected components, and without {\em a priori} information about the locations of the positive and negative inclusions. In the definite case the method gives a criterion on whether a given open set is enclosed by the (outer shape of the) inclusions or not; in the indefinite case the roles are reversed and the criterion is on whether a given set encloses the inclusions or not. By applying either of these criteria to various test sets, the outer shape of the inclusions can be reconstructed. Furthermore, due to the simplicity of the method, it adapts in a natural way to more complicated and realistic electrode models allowing a straightforward numerical implementation \cite{GardeStaboulis_2016,Garde_2019,Harrach15,Garde_2017a}. The ideas behind the monotonicity method have also been used to obtain stability estimates from finitely many measurements \cite{Harrach_2019}, and the method has very recently been applied to the reconstruction of piecewise constant conductivities with a layered structure \cite{Garde_2019b}. 
	
	For \emph{extreme} inclusions, meaning that the conductivity formally takes the value zero (perfectly insulating) or infinity (perfectly conducting) in some parts of the domain, there exist variants of both the factorization method (e.g.~\cite[Section 4.3.1]{BruhlPhD} and \cite[Section 2.2.2]{SchmittPhD}) and the enclosure method \cite{Ikehata1999a,Brander_2018}. Furthermore, there are a number of reconstruction methods designed specifically for extreme inclusions; see \cite{Friedman_1989,Munnier_2016,Kress_2005,Kress_2012,Borman_2009,Moradifam_2012} for information on some of them. A common feature of all these methods is that they have not been proven to be functional if the examined body is simultaneously contaminated by both perfectly conducting and perfectly insulating inclusions.
	
	The monotonicity method has not previously been generalized to allow extreme inclusions inside the domain of interest; fixing this flaw is the main topic of this paper. The main obstacle for such a generalization has arguably been establishing suitable monotonicity principles for the ND map in the presence of extreme inclusions. To this end, we first show that any ND map corresponding to extreme inclusions can be reached in the natural operator topology as a limit of a sequence of standard ND maps when the conductivity coefficient decays to zero and/or grows to infinity in the appropriate parts of the domain. In fact, the convergence rate equals the square root of the truncation parameter $\epsilon > 0$ that via multiplication/division controls the convergence of the conductivity to zero/infinity. Utilizing the standard monotonicity relations for the ND maps corresponding to the truncated conductivities, we show that both definite and indefinite inclusions can be reconstructed with the monotonicity method in the presence of  extreme inclusions, that is, we obtain natural generalizations of the results in \cite{Harrach13,Garde_2019}. In particular, our result on the indefinite case covers a wide class of inclusions: the perturbed conductivity can simultaneously have parts that are perfectly insulating, parts that are perfectly conducting, as well as other perturbed parts bounded away from zero and infinity.
	
	This article is organized as follows. Section~\ref{sec:continuum} gives a brief introduction to the mathematical formulation of the forward problem of EIT in the presence of extreme inclusions. Section~\ref{sec:main} lists the main results of the paper, while Sections~\ref{sec:NDconv}--\ref{sec:monomethodindefinite} are dedicated to the proofs of these main results. We also include Appendix~\ref{sec:appendixA} on some generic monotonicity estimates for ND maps with various combinations of extreme inclusions. Appendix~\ref{sec:appendixB} gives an example on different limiting behaviors of potentials in the perfectly insulating parts of the domain.
	
	To be somewhat more precise, Section~\ref{sec:NDconv} proves the convergence of the electric potential corresponding to a conductivity bounded away from zero and infinity toward its counterpart with extreme inclusions as the conductivity decays to zero and grows to infinity in the appropriate parts of the domain.  The actual result is stated in Theorem~\ref{thm:uconverge}, and it also covers the convergence of the associated ND maps.  Sections~\ref{sec:monomethoddefinite} and~\ref{sec:monomethodindefinite} prove the functionality of the monotonicity method in a framework that also allows extreme inclusions. In the definite case, the result for the linearized version of the method is given in Theorem~\ref{thm:monofast} and for the nonlinear version in Theorem~\ref{thm:monoslow}. Our main result on the characterization of indefinite inclusions, incorporating the extreme cases, is given as Theorem~\ref{thm:general}. 
	
	\subsection{Notational remarks}
	
	Let $\mathscr{L}(X,Y)$ denote the space of bounded linear operators between Banach spaces $X$ and $Y$, with $\mathscr{L}(X) := \mathscr{L}(X,X)$. For a Hilbert space $X$ and self-adjoint operators $A,B\in \mathscr{L}(X)$, we write $A\geq B$ if $A-B$ is positive semidefinite. 
	
 We denote the interior of a set $E\subseteq\rum{R}^d$ by $E^\circ$ and its characteristic function by $\chi_E$. For a real measurable essentially bounded function $\varsigma$, we write $\inf_E(\varsigma)$ and $\sup_E(\varsigma)$ in place of $\essinf(\varsigma|_E)$ and $\esssup(\varsigma|_E)$, and avoid the subscript notation if no restriction is used. 
	
	$K$ will denote a generic positive constant.
	
	\section{The continuum model}
	\label{sec:continuum}
	
	Assume $\Omega\subset \rum{R}^d$, $d\geq 2$, is a bounded Lipschitz domain with $\rum{R}^d\setminus\overline{\Omega}$ connected. Let $\Gamma\subseteq \partial\Omega$ be a relatively open nonempty boundary piece; $\inner{\cdot,\cdot}$ denotes the associated inner product on $L^2(\Gamma)$. For any measurable set $\tilde{\Omega}\subseteq\Omega$, we define
	\begin{equation*}
	L^\infty_+(\tilde{\Omega}) := \{ \varsigma\in L^\infty(\tilde{\Omega};\rum{R}) \mid \inf(\varsigma) > 0 \}.
	\end{equation*}
	If $\tilde{\Omega}$ is a Lipschitz domain with $\Gamma\subseteq \partial\tilde{\Omega}$, we define the ``$\Gamma$-mean free spaces'':
	\begin{align*}
	H^1_\diamond(\tilde{\Omega}) &:= \{ w\in H^1(\tilde{\Omega}) \mid \inner{1,w|_\Gamma} = 0\}, \\
	L^2_\diamond(\Gamma) &:= \{ g\in L^2(\Gamma) \mid \inner{1,g} = 0 \}.
	\end{align*}
	
	A set $C_0$ will represent a part of the domain that is perfectly insulating (zero conductivity) and $C_\infty$ will represent a part of the domain that is perfectly conducting (infinite conductivity). We assume these extreme inclusions satisfy the following conditions.
	\begin{assumption} \label{assump}
		Let $C = C_0\cup C_\infty$, where $C_0,C_\infty\subset \Omega$ satisfy:
		\begin{enumerate}[(i)]
			\item $C_0$ and $C_\infty$ are closures of open sets with finitely many components and Lipschitz boundaries.
			\item $C_0 \cap C_\infty = \emptyset$. 
			\item $\Omega\setminus C_0$ is connected.
		\end{enumerate}
	\end{assumption}
	\noindent Note that $C_0$ and $C_\infty$ can consist of several connected components, and the sets are also allowed to be empty.
	
	For $\varsigma \in L^\infty_+(\Omega)$, $\sigma = \sigma(\varsigma,C_0,C_\infty)$ denotes the nonnegative conductivity 
	\begin{equation*} 
	\sigma := \begin{cases}
	\varsigma & \text{in } \Omega\setminus C, \\
	0 & \text{in } C_0, \\
	\infty & \text{in } C_\infty.
	\end{cases}
	\end{equation*}
	Throughout this work we are interested in the following elliptic boundary value problem:
	\begin{align*}
	-\nabla \cdot(\sigma \nabla u) &= 0 \text{ in } \Omega \setminus C, \\[1mm]
	\nu\cdot(\sigma\nabla u) &= \begin{cases}
	f & \text{on } \Gamma, \\
	0 & \text{on } \partial (\Omega\setminus C_0)\setminus\Gamma,
	\end{cases} \\[1mm]
	\nabla u &=0 \text{ in } C_\infty^\circ, \\
	\int_{\partial C_i} \nu\cdot(\sigma\nabla u)\,\di S &= 0 \text{ for each component $C_i$ of $C_\infty$,}
	\end{align*}
	where  $\nu$ is the outer unit normal to $\Omega\setminus C$ and $f\in L_\diamond^2(\Gamma)$ represents an applied current density on $\Gamma$, while $u$ is the corresponding electric potential in $\Omega \setminus C_0$. In the last condition, the trace of $\nu\cdot(\sigma\nabla u)$ is taken from the exterior of $C_\infty$.
	
	Formally, we write the above problem as
	\begin{equation} \label{eq:continuum}
	-\nabla\cdot(\sigma\nabla u) = 0 \text{ in } \Omega, \qquad \nu\cdot(\sigma\nabla u) = \begin{cases}
	f & \text{on } \Gamma, \\
	0 & \text{on } \partial \Omega\setminus\Gamma,
	\end{cases}
	\end{equation}
        where the values $0$ and $\infty$ of $\sigma$ are to be interpreted via the appropriate boundary conditions for $C_0$ and $C_\infty$, and $u$ is only defined in $\Omega\setminus C_0$. If $C_0 = C_\infty = \emptyset$, then \eqref{eq:continuum} reduces to the standard conductivity equation for a conductivity coefficient in $L^\infty_+(\Omega)$. 
	
	Because of Assumption~\ref{assump}(iii), we may equip $H^1_\diamond(\Omega\setminus C_0)$ with either of the inner products
	\begin{align}
	\inner{v,w}_\varsigma &:= \int_{\Omega\setminus C_0} \varsigma \nabla v\cdot\nabla \overline{w}\,\di x, \label{eq:innerprodsigma} \\
	\inner{v,w}_* &:= \inner{\nabla v,\nabla w}_{L^2(\Omega\setminus C_0; \rum{C}^d)}. \notag
	\end{align}
	The corresponding norms $\norm{\cdot}_\varsigma$ and $\norm{\cdot}_*$ are equivalent to the usual $H^1$-norm on $H^1_\diamond(\Omega\setminus C_0)$, due to a Poincar\'e-type inequality related to the zero-mean condition on $\Gamma$:
	\begin{equation*}
	\frac{1}{\sup(\varsigma)^{1/2}}\norm{v}_\varsigma \leq \norm{v}_* \leq \norm{v}_{H^1(\Omega\setminus C_0)} \leq K\norm{v}_* \leq \frac{K}{\inf(\varsigma)^{1/2}}\norm{v}_\varsigma, \quad v\in H^1_\diamond(\Omega\setminus C_0). 
	\end{equation*}
	We will often make use of the above equivalences, and will do so without explicitly mentioning it.
	
	In order to rigorously define the variational problem associated to \eqref{eq:continuum}, we introduce the function space
	\begin{equation*}
		\mathscr{H}_\diamond = \mathscr{H}_\diamond(C_0,C_\infty) := \{ w \in H^1_\diamond(\Omega\setminus C_0) \mid \nabla w = 0 \text{ in } C_\infty^\circ \}.
	\end{equation*}
	It follows straightforwardly that $\mathscr{H}_\diamond$ is a closed subspace (thus itself a Hilbert space) of $H_\diamond^1(\Omega\setminus C_0)$ when equipped with either of the above inner products. In particular, we denote by $P(\varsigma,C_0,C_\infty)$ the orthogonal projection of $H^1_\diamond(\Omega\setminus C_0)$ onto $\mathscr{H}_\diamond$ in the inner product $\inner{\cdot,\cdot}_\varsigma$ (cf.~\eqref{eq:innerprodsigma}), while $P_\perp(\varsigma,C_0,C_\infty)$ is the orthogonal projection onto the orthogonal complement $\mathscr{H}_\diamond^\perp$. 

	Employing the Lax--Milgram lemma, it easily follows that \eqref{eq:continuum} has a unique weak solution $u = u_f^\sigma$ in $\mathscr{H}_\diamond$ satisfying
	\begin{equation}
	\int_{\Omega\setminus C} \sigma \nabla u \cdot \nabla \overline{v}\,\di x = \inner{f,v|_\Gamma}, \quad \forall v\in \mathscr{H}_\diamond. \label{eq:weakform}
	\end{equation} 
	This furthermore gives the estimate
	\begin{equation*}
	\norm{u}_{H^1(\Omega\setminus C_0)} \leq \frac{K}{\inf(\varsigma)}\norm{f}_{L^2(\Gamma)}. 
	\end{equation*}
	\begin{remark} \label{remark:weakform}
		To obtain \eqref{eq:weakform} from \eqref{eq:continuum}, we note that introducing a perfectly insulating inclusion to such a Neumann problem corresponds to starting out with a smaller domain $\Omega\setminus C_0$ to begin with. The requirement $\nabla u = 0$ in $C_\infty^\circ$ is naturally introduced via the choice of the variational space $\mathscr{H}_\diamond$. More precisely, deducing \eqref{eq:weakform} is a matter of integration by parts in $\Omega\setminus C$, which leaves no boundary terms related to $\partial C_0$ or $\partial C_\infty$ due to the homogeneous Neumann condition on $\partial C_0$ and because $\nu \cdot (\sigma \nabla u)$ is mean free and the test function $v \in \mathscr{H}_\diamond$ is constant on the boundary of each component in $C_\infty$. On the other hand, the existence and uniqueness of a solution for \eqref{eq:weakform} implies that the weak solution coincides with the strong solution, whenever the latter exists.

    	It is also straightforward to see that the solution to \eqref{eq:weakform} always satisfies \eqref{eq:continuum} in a suitable weak sense. Indeed, as $C_{\textup{c}}^\infty(\Omega\setminus C) \subset \mathscr{H}_\diamond$ via zero continuation to $C_\infty$, the solution of \eqref{eq:weakform} satisfies the conductivity equation in the distributional sense in $\Omega \setminus C$. Moreover, the `boundary condition' in $C_\infty$ is automatic due to the choice of the variational space $\mathscr{H}_\diamond$, and the Neumann conditions on $\partial \Omega$ and $\partial C_0$ hold in an appropriate sense of traces based on suitable Green's formulas (cf.~\cite{Lions1972,Necas2012}).
	\end{remark}
	\begin{remark}
		Throughout this article, we will use the notation $u_f^\sigma$ to indicate the solution to \eqref{eq:weakform} for Neumann condition $f$ and nonnegative conductivity $\sigma$. Similar notation will also be used for conductivity coefficients with other choices of background conductivity, perfectly conducting parts, and perfectly insulating parts. In particular, for $\varsigma\in L^\infty_+(\Omega)$ we note that $\varsigma = \sigma(\varsigma,\emptyset,\emptyset)$, and therefore $u_f^\varsigma\in H^1_\diamond(\Omega)$ satisfies
	\begin{equation*}
		\int_{\Omega} \varsigma \nabla u_f^\varsigma \cdot \nabla \overline{v}\,\di x = \inner{f,v|_\Gamma}, \quad \forall v\in H^1_\diamond(\Omega).
	\end{equation*} 
	\end{remark}
	
	Another useful characterization of $u_f^\sigma$ relates it to the inclusion-free potential $u_f^\varsigma$.
	\begin{proposition} \label{prop:altcharacterization}
		Let $C = C_0\cup C_\infty$ satisfy Assumption~\ref{assump}, $\sigma = \sigma(\varsigma,C_0,C_\infty)$, and $P = P(\varsigma,C_0,C_\infty)$. For any $f\in L^2_\diamond(\Gamma)$, $u_f^\sigma$ can be uniquely characterized as follows: 
		\begin{equation*}
			u_f^\sigma = P(u_f^\varsigma|_{\Omega\setminus C_0} - w),
		\end{equation*}
		where $w$ is the unique solution in $H^1_\diamond(\Omega\setminus C_0)$ to
		\begin{align*}
			-\nabla\cdot(\varsigma\nabla w) &= 0 \text{ in } \Omega \setminus C_0, \\
			\nu\cdot (\varsigma \nabla w) &= \nu\cdot(\varsigma \nabla u_f^\varsigma|_{\Omega\setminus C_0}) \text{ on } \partial C_0, \\
			\nu\cdot (\varsigma \nabla w) &= 0 \text{ on } \partial\Omega.
		\end{align*}
	\end{proposition}
	\begin{proof}
		It is clear from the boundary value problems defining $u_f^\varsigma$ and $w$ that $u_f^{\hat{\sigma}} = u_f^\varsigma|_{\Omega\setminus C_0} - w$ where $\hat{\sigma} = \sigma(\varsigma,C_0,\emptyset)$. Recall that $P$ is an orthogonal projection and hence self-adjoint with respect to the inner product \eqref{eq:innerprodsigma}. Due to the weak formulation for $u_f^{\hat{\sigma}}$, the following thus holds for all $v = Pv\in \mathscr{H}_\diamond(C_0,C_\infty)$:
		\begin{equation*}
			\inner{f,v|_\Gamma} = \int_{\Omega\setminus C_0} \varsigma \nabla u_f^{\hat{\sigma}} \cdot \nabla \overline{Pv} \,\di x = \int_{\Omega\setminus C} \varsigma \nabla Pu_f^{\hat{\sigma}} \cdot \nabla \overline{v} \,\di x.
		\end{equation*}
		By the unique solvability of the variational problem \eqref{eq:weakform}, we conclude that $u_f^\sigma=Pu_f^{\hat{\sigma}}$.
	\end{proof}
	Finally, we define the local Neumann-to-Dirichlet (ND) map as $\Lambda(\sigma) : f\mapsto u_f^\sigma|_\Gamma$, which is a compact self-adjoint operator in $\mathscr{L}(L^2_\diamond(\Gamma))$. 
	
	\section{Main results}
	\label{sec:main}
	
	Our first result is on retrieving the electric potential, and thereby also the ND map, in the case of perfectly insulating and perfectly conducting inclusions via a limit of a sequence of truncated conductivities that belong to $L^\infty_+(\Omega)$. This theorem will be used for the proofs of all the other main results.
	
	\begin{theorem} \label{thm:uconverge}
		Let $C = C_0\cup C_\infty$ satisfy Assumption~\ref{assump}, $\sigma = \sigma(\varsigma,C_0,C_\infty)$, and define the $\epsilon$-truncated version of $\sigma$, with $\epsilon>0$, by
		\begin{equation*} 
		\sigma_\epsilon := \begin{cases}
		\varsigma & \textup{in } \Omega\setminus C, \\
		\epsilon\varsigma & \textup{in } C_0, \\
		\epsilon^{-1}\varsigma & \textup{in } C_\infty.
		\end{cases}
		\end{equation*}
		Then the following estimate holds
		\begin{equation}
		\norm{u_f^{\sigma_\epsilon} - u_f^\sigma}_{H^1(\Omega\setminus C_0)} \leq K \epsilon^{1/2} \norm{f}_{L^2(\Gamma)}, \label{eq:convrate}
		\end{equation}
        with $K>0$ independent of $f$ and $\epsilon$.
		As a direct consequence
		\begin{equation*}
		\norm{\Lambda(\sigma_\epsilon) - \Lambda(\sigma)}_{\mathscr{L}(L^2_\diamond(\Gamma))} \leq K\epsilon^{1/2}.
		\end{equation*}
	\end{theorem}
	As a side note for Theorem~\ref{thm:uconverge}, while there exist several $H^1$-extensions of $u_f^\sigma$ into $C_0$, there is a natural extension that coincides with the limiting behavior of $u_f^{\sigma_\epsilon}|_{C_0}$ for $\epsilon\to 0$.
	\begin{corollary} \label{coro:extended}
		Let $C$, $\sigma$, and $\sigma_\epsilon$ be as in Theorem~\ref{thm:uconverge}. Consider the $H^1$-extension operator $E = E(\varsigma,C_0) : H^1(\Omega\setminus C_0) \to H^1(\Omega)$, defined for $w\in H^1(\Omega\setminus C_0)$ as
		\begin{equation*}
		E w := \begin{cases}
		w & \textup{in } \Omega\setminus C_0, \\
		\tilde{w} & \textup{in } C_0,
		\end{cases}
		\end{equation*}
		where $\tilde{w}$ is the unique solution in $H^1(C_0^\circ)$ to the Dirichlet problem
		\begin{equation*} 
		-\nabla\cdot(\varsigma\nabla \tilde{w}) = 0 \text{ in } C_0^\circ, \qquad \tilde{w} = w \text{ on } \partial C_0.
		\end{equation*}
		Then it holds
		\begin{equation*}
		\norm{u_f^{\sigma_{\epsilon}} - Eu_f^\sigma}_{H^1(\Omega)} \leq K\epsilon^{1/2}\norm{f}_{L^2(\Gamma)}.
		\end{equation*}
	\end{corollary}
	\begin{proof}
		Since $\nabla\cdot(\varsigma \nabla u_f^{\sigma_{\epsilon}}) = \nabla\cdot(\sigma_\epsilon \nabla u_f^{\sigma_{\epsilon}}) = 0$ in $C_0^\circ$, the restriction $v := (u_f^{\sigma_{\epsilon}} - E u_f^\sigma)|_{C_0}$ satisfies the Dirichlet problem
		\begin{equation*} 
		-\nabla\cdot(\varsigma\nabla v) = 0 \text{ in } C_0^\circ, \qquad v = u_f^{\sigma_{\epsilon}} - u_f^\sigma \text{ on } \partial C_0.
		\end{equation*}
		Hence, by the continuous dependence on the Dirichlet boundary value, the trace theorem in $\Omega\setminus C_0$, and the result of Theorem~\ref{thm:uconverge}, we have
		\begin{align*}
		\norm{v}_{H^1(C_0^\circ)} \leq K\norm{(u_f^{\sigma_{\epsilon}} - u_f^\sigma)|_{\partial C_0}}_{H^{1/2}(\partial C_0)} \leq K\norm{u_f^{\sigma_{\epsilon}} - u_f^\sigma}_{H^1(\Omega\setminus C_0)} \leq K\epsilon^{1/2}\norm{f}_{L^2(\Gamma)}.
		\end{align*}
		This estimate, combined with the result of Theorem~\ref{thm:uconverge}, concludes the proof.
	\end{proof}
	
	Next we will consider the reconstruction of inclusions via the monotonicity method. For this purpose we introduce a family of admissible test inclusions as 
	\begin{align*}
		\mathcal{A} &:= \{ C \subset \Omega \mid C \text{ is the closure of an open set,}  \\
		&\hphantom{{}:= \{C \subset \Omega \mid{}}\text{has connected complement,} \\
		&\hphantom{{}:= \{C \subset \Omega \mid{}}\text{and has Lipschitz boundary } \partial C \}.
	\end{align*}
	Before discussing the more general indefinite case, we first assume $D\in \mathcal{A}$ represents either perfectly insulating or perfectly conducting inclusions that we seek to recover. These partial results are stated as separate theorems because they allow to recover $D$ based on positive semidefiniteness tests in comparison to conductivities that belong to $L^\infty_+(\Omega)$ (see~Theorem~\ref{thm:monoslow}), which is straightforward from the standpoint of a numerical implementation. We also introduce a linearized version of the reconstruction method that only relies on the ND map and its Fr\'echet derivative at the background conductivity, hence leading to a very fast algorithm.
	
	Let $\gamma_0\in L^\infty_+(\Omega)$ denote the so-called background conductivity and let $C\in \mathcal{A}$. We use the following notation:
	\begin{enumerate}[(i)]
		\item $\DLambda(\gamma_0;\,\cdot\,) \in \mathscr{L}(L^\infty(\Omega;\rum{R}),\mathscr{L}(L^2_\diamond(\Gamma)))$ denotes the Fr\'echet derivative of $\Lambda|_{L^\infty_+(\Omega)}$ evaluated at $\gamma_0$, characterized by the well-known quadratic formula
		\begin{equation} \label{eq:DNfrechet}
		\inner{\DLambda(\gamma_0;\eta)f,f} = -\int_{\Omega} \eta \abs{\nabla u_f^{\gamma_0}}^2\,\di x, \quad \eta\in L^\infty(\Omega;\rum{R}), \; f\in L^2_\diamond(\Gamma).
		\end{equation}
		\item $\Lambda_0(C) := \Lambda(\sigma(\gamma_0,C,\emptyset))$ denotes the ND map for the conductivity characterized by perfectly insulating inclusions in $C$ and $\gamma_0$ in $\Omega\setminus C$.
		\item $\Lambda_\infty(C) := \Lambda(\sigma(\gamma_0,\emptyset,C))$ denotes the ND map for the conductivity characterized by perfectly conducting inclusions in $C$ and $\gamma_0$ in $\Omega\setminus C$.
	\end{enumerate}
	We will furthermore require that $\gamma_0$ satisfies a unique continuation principle.
	\begin{definition} \label{defi:ucp}
		Suppose $U\subseteq \overline{\Omega}$ is a relatively open and connected set that intersects $\Gamma$. We say that $\varsigma\in L^\infty_+(\Omega)$ satisfies the weak \emph{unique continuation principle} (UCP) in $U$ for the conductivity equation, provided that only the trivial solution of 
		\begin{equation*}
			-\nabla \cdot(\varsigma\nabla v) = 0 \text{ in } U^\circ
		\end{equation*}
		can be identically zero in a nonempty open subset of $U$, and likewise only the trivial solution has vanishing Cauchy data on $\partial U\cap \Gamma$. If $U = \overline{\Omega}$ we simply say that $\varsigma$ satisfies the UCP.
	\end{definition}
	We are now ready to introduce the linearized and nonlinear reconstruction methods for recovering $D$ based on boundary data given either as $\Lambda_0(D)$ or $\Lambda_\infty(D)$.
	\begin{theorem} \label{thm:monofast}
		Let $\gamma_0\in L^\infty_+(\Omega)$ satisfy the UCP, let $0 < \beta \leq \inf(\gamma_0)$, and let $B\subseteq\Omega$ be a nonempty open set. For $D\in\mathcal{A}$, it holds
		\begin{align*}
		B\subset D &\qquad \text{if and only if} \qquad \Lambda_0(D) \geq \Lambda(\gamma_0)-\beta \DLambda(\gamma_0;\chi_B) \\
		&\qquad \text{if and only if} \qquad \Lambda(\gamma_0) + \beta \DLambda(\gamma_0;\chi_B) \geq \Lambda_\infty(D).
		\end{align*}
	\end{theorem}
	\begin{theorem} \label{thm:monoslow}
		Let $\gamma_0\in L^\infty_+(\Omega)$ satisfy the UCP and let $B\subseteq\Omega$ be a nonempty open set. For $D\in\mathcal{A}$, it holds:
		\begin{alignat*}{2}
		\text{If } 0 < \beta < \inf(\gamma_0), \text{ then} &\qquad B\subset D& \qquad &\text{if and only if} \qquad \Lambda_0(D) \geq \Lambda(\gamma_0-\beta\chi_B). \\
		\text{If } \beta >0, \text{ then} &\qquad B\subset D&  &\text{if and only if} \qquad \Lambda(\gamma_0 + \beta\chi_B) \geq \Lambda_\infty(D).
		\end{alignat*}
	\end{theorem}
	Finally, we consider the setting where the conductivity may be both larger and smaller than the background conductivity $\gamma_0$, and it may also have parts that are perfectly insulating and parts that are perfectly conducting. 
	\begin{definition}[indefinite inclusion] \label{def:indefinite}
		Suppose a set $D\in\mathcal{A}$ can be decomposed as
		\begin{equation*}
			D = D^+ \cup D^-, \qquad D^- = D_\textup{F}^- \cup D_0, \qquad \text{and} \qquad D^+ = D_\textup{F}^+ \cup D_\infty,
		\end{equation*}
		where we assume:
		\begin{enumerate}[(i)]
			\item $D_0\cup D_\infty$ satisfies Assumption~\ref{assump}, with $D_0$ in place of $C_0$ and $D_\infty$ in place of $C_\infty$.
			\item $D_{\textup{F}}^{\pm}$ are measurable and mutually disjoint with the set $D_0\cup D_\infty$.
			\item For every open neighborhood $W$ of $x\in \partial D$, there exists a relatively open set $V\subset D$ that intersects $\partial D$ and either $V \subset D^-\cap W$ or $V\subset D^+ \cap W$. \label{indefiniteassump4}
		\end{enumerate}
		Let $\gamma_0\in L^\infty_+(\Omega)$ satisfy the UCP and define $0\leq \gamma \leq \infty$ by
		\begin{equation*}
		\gamma := \begin{cases}
		0 & \text{in } D_0, \\
		\infty & \text{in } D_\infty, \\
		\gamma_- & \text{in } D_\textup{F}^-, \\
		\gamma_+ & \text{in } D_\textup{F}^+, \\
		\gamma_0 & \text{in } \Omega\setminus D,
		\end{cases}
		\end{equation*}
		where $\gamma_{\pm}\in L^\infty_+(D_\textup{F}^\pm)$ satisfy $\sup_{D_\textup{F}^-}(\gamma_--\gamma_0)<0$ and $\inf_{D_\textup{F}^+}(\gamma_+-\gamma_0)>0$, i.e.\ $\gamma$ is essentially bounded away from $\gamma_0$ in $D$. Then $D$ is called an \emph{indefinite inclusion} with respect to the background conductivity $\gamma_0$.
	\end{definition}
	This prelude leads to the following general approach to finding inclusions in EIT.
	\begin{theorem} \label{thm:general}
		Let $\gamma$ and $D$ be as in Definition~\ref{def:indefinite}. For $C\in\mathcal{A}$, it holds
		\begin{equation*}
		D\subseteq C \qquad \text{if and only if} \qquad \Lambda_0(C) \geq \Lambda(\gamma) \geq \Lambda_\infty(C).
		\end{equation*}
		In particular, $D = \cap\{ C\in\mathcal{A} \mid \Lambda_0(C) \geq \Lambda(\gamma) \geq \Lambda_\infty(C) \}$.
	\end{theorem}	
	
	\subsection{Additional remarks on the main results} \label{sec:remarks}
	
	This subsection provides some additional comments, examples, and extensions on the definitions and results in Section~\ref{sec:main}.
	
	\subsubsection*{On Theorem~\ref{thm:uconverge}:}
	\begin{enumerate}[(i)]
	\item We only insist on a positive distance between $\partial C_0$, $\partial C_\infty$, and $\partial \Omega$ for conceptual and notational convenience in Theorem~\ref{thm:uconverge} and its proof. In fact, one could deduce the variational form \eqref{eq:weakform} and the estimates in the proof of Theorem~\ref{thm:uconverge} by replacing the conditions $C_0,C_\infty\subset \Omega$ and (ii) in Assumption~\ref{assump} by the following weaker ones: $C_0,C_\infty\subset \overline{\Omega}$, $C_0^\circ \cap C_\infty^\circ = \emptyset$, $C\cap \Gamma = \emptyset$, and $C$ and $\Omega\setminus E$ have Lipschitz boundaries for $E\in\{C_0,C_\infty,C\}$. Under these weaker assumptions, the boundary conditions in the original boundary value problem \eqref{eq:continuum} need to be slightly modified: the current density $\nu \cdot (\sigma \nabla u)$ vanishes on $\partial (\Omega \setminus C_0) \setminus (\Gamma \cup \partial C_\infty)$ and its integral is zero over the part of $\partial C_i$ that touches $\Omega \setminus C$.          
 		\item For the sake of keeping the presentation of Theorem~\ref{thm:uconverge} and its proof clean, we have not included the precise dependence of the constant $K$ on $\varsigma$. Actually, $K$ only depends on $\varsigma$ as a multiple of a positive power of $\sup(\varsigma)$ and a negative power of $\inf(\varsigma)$. Furthermore, $K$ depends on $\Omega$, $C_0$, and $C_\infty$ via the constants that appear in the corresponding trace theorems and Poincar\'e inequalities. \label{remark:converge1}
		\item Let $(\epsilon_n)\subset (0,\infty)$ with $\lim_{n\to \infty} \epsilon_n = 0$. Consider a sequence $(\sigma_n)\subset L^\infty_+(\Omega)$ satisfying $\sigma_n \equiv \varsigma$ in $\Omega\setminus C$ and
		%
		\begin{align*}
		K^{-1}\epsilon_n \leq \sigma_n &\leq K \epsilon_n \textup{ in } C_0, \\
		K^{-1}\epsilon_n^{-1} \leq \sigma_n &\leq K \epsilon_n^{-1} \textup{ in } C_\infty
		\end{align*}
                for some constant $K\geq 1$.
		Then Theorem~\ref{thm:uconverge} and item \eqref{remark:converge1} above imply the convergence of the associated potentials in $H^1_\diamond(\Omega\setminus C_0)$ and of the corresponding ND maps in the operator norm of $\mathscr{L}(L^2_\diamond(\Gamma))$, both with the rate $\epsilon_n^{1/2}$.
		\item Consider a sequence $(\sigma_n)\subset L^\infty_+(\Omega)$ satisfying $\sigma_n \equiv \varsigma$ in $\Omega\setminus C_0$ and 
		\begin{equation} \label{eq:limitcond1}
		\lim_{n\to \infty}\sup_{C_0}(\sigma_n) \to 0.
		\end{equation}
		Theorem~\ref{thm:uconverge} with $C_\infty=\emptyset$ and a truncated conductivity coefficient with the constant value $\sup_{C_0}(\sigma_n)$ in $C_0$, combined with the monotonicity principles in Lemma~\ref{lemma:monoappendix} and the squeeze theorem, demonstrates that $\Lambda(\sigma_n)$ converges in the operator norm. The same conclusion also holds if instead $\sigma_n\equiv \varsigma$ in $\Omega\setminus C_\infty$ and
		\begin{equation} \label{eq:limitcond2}
		\lim_{n\to \infty}\inf_{C_\infty}(\sigma_n)\to\infty.
		\end{equation}
		Hence, in the `one-sided' cases with \emph{either} decay to zero in $C_0$ {\em or} growth to infinity in $C_\infty$, the spatial variations in the decay (or growth) rate within the region $C_0$ (or $C_\infty$) are irrelevant for the corresponding ND map to converge in the operator norm, as long as the convergence of the conductivity coefficient is uniform in the sense of \eqref{eq:limitcond1} and \eqref{eq:limitcond2}. \label{remark:NDconverge}
	      \item It is not obvious whether a convergence result like the one presented in item \eqref{remark:NDconverge} holds if the sequence of conductivity coefficients converges (necessarily in a nonmonotonic way) to a limit conductivity with both perfectly insulating and perfectly conducting inclusions. However, by considering two iterated limits, one for $C_0$ and one for $C_\infty$, the ND maps converge in the operator norm if conditions of the form \eqref{eq:limitcond1} and \eqref{eq:limitcond2} are satisfied and the conductivities in the double-sequence equal $\varsigma$ in $\Omega\setminus C$. Unfortunately, this does not guarantee that one can pass to the limit simultaneously in both $C_0$ and $C_\infty$ with spatially varying rates.
	\end{enumerate}
	
	\subsubsection*{On Corollary~\ref{coro:extended}:}
	\begin{enumerate}[(i)]
		\setcounter{enumi}{5}
		\item In general, the limiting behavior of an electric potential in $C_0$ is ambiguous, as it depends on spatial variations in the decay rate for the conductivity in $C_0$. Hence, to obtain the simple characterization in Corollary~\ref{coro:extended}, it is required that the decay rate is the same throughout $C_0$. In Appendix~\ref{sec:appendixB} we give an example illustrating this ambiguity by letting the decay rate of the conductivity be governed by $\epsilon$ in some parts of $C_0$ and by $\epsilon^2$ in other parts. As a side note, the form of the limit potential inside $C_\infty$ is unambiguous as it is uniquely defined by the variational problem \eqref{eq:weakform}.
	\end{enumerate}
	
	\subsubsection*{On Definition~\ref{defi:ucp}:}
	\begin{enumerate}[(i)]
		\setcounter{enumi}{6}
		\item As a simple example, a piecewise analytic conductivity coefficient $\varsigma \in L^\infty_+(\Omega)$ satisfies the UCP; see~e.g.~\cite{Kohn1985} for the case $\partial \Omega \in \mathscr{C}^\infty$.
	\end{enumerate}

	\subsubsection*{On Definition~\ref{def:indefinite}:}
	\begin{enumerate}[(i)]
		\setcounter{enumi}{7}
		\item The last assumption \eqref{indefiniteassump4} in Definition~\ref{def:indefinite} implies the reasonable property that the sign of $\gamma-\gamma_0$ cannot change arbitrarily often near $\partial D$; see e.g.~\cite[Remark~2(iv)]{Garde_2019} for an example of such behavior. This condition is required to guarantee the existence of relatively open subsets of $D$ near $\partial D$ that only belong to either $D^+$ or $D^-$. Moreover, the Lipschitz regularity of $\partial D_0$ and $\partial D_\infty$ guarantees that near $\partial D$ we can choose relatively open subsets of $D$  that only belong to one of the subregions $D_\textup{F}^+$, $D_\textup{F}^-$, $D_0$, and $D_\infty$.
	\end{enumerate}

	\subsubsection*{On Theorem~\ref{thm:monofast}, Theorem~\ref{thm:monoslow}, and Theorem~\ref{thm:general}:}
	\begin{enumerate}[(i)]
		\setcounter{enumi}{8}
		\item If $\mathcal{B}$ is the set of open balls in $\Omega$, for the appropriate values of $\beta$ in Theorem~\ref{thm:monofast} and Theorem~\ref{thm:monoslow}, we have 
		\begin{align*}
		D^\circ &= \cup\{ B\in\mathcal{B} \mid \Lambda_0(D) \geq \Lambda(\gamma_0)-\beta \DLambda(\gamma_0;\chi_B) \} \\
		&= \cup\{ B\in\mathcal{B} \mid \Lambda_0(D) \geq \Lambda(\gamma_0-\beta\chi_B) \} \\
		&= \cup\{ B\in\mathcal{B} \mid \Lambda(\gamma_0) + \beta \DLambda(\gamma_0;\chi_B) \geq \Lambda_\infty(D) \} \\
		&= \cup\{ B\in\mathcal{B} \mid \Lambda(\gamma_0 + \beta\chi_B) \geq \Lambda_\infty(D) \}.
		\end{align*}
		Based on these characterizations, we readily obtain $D$ as the closure of $D^\circ$, if necessary. In a numerical implementation, one may wish to replace ``balls'' by ``pixels/voxels/triangulation'' or other appropriate open sets.
		\item If the set $D$ does not have a connected complement, there are parts of $\partial D$ that cannot be reached by the localized potentials in Lemma~\ref{lemma:locpot} and Lemma~\ref{lemma:locpot2} without $U$ intersecting $D$, and one would only obtain a reconstruction up to the outer shape of $D$; see e.g.\ \cite{Harrach13,Garde_2019}.
		\item According to the proof of Theorem~\ref{thm:general}, its result also applies to the definite cases in the following way: if $D^+ = \emptyset$, we only need to consider the inequality $\Lambda_0(C)\geq \Lambda(\gamma)$, and if $D^- = \emptyset$, we only need to consider $\Lambda(\gamma)\geq \Lambda_\infty(C)$.
		\item Note that in all the extreme cases, we assume that the background conductivity remains essentially bounded away from zero and infinity. It remains an open problem to figure out how to employ the monotonicity method if the conductivity tends to zero (a degenerate problem) or tends to infinity (a singular problem) near the inclusion boundary.
	\end{enumerate}
	
	\section{Proof of Theorem~\ref{thm:uconverge}} \label{sec:NDconv}
	
	For sake of brevity, we write $u = u_f^\sigma$ and $u_\epsilon = u_f^{\sigma_\epsilon}$. Recall that $\mathscr{H}_\diamond = \mathscr{H}_\diamond(C_0,C_\infty)$ is the function space related to $u$; the space related to $u_\epsilon$ is $H^1_\diamond(\Omega)$ since the corresponding conductivity is bounded away from zero and infinity in $\Omega$. As in \eqref{eq:innerprodsigma}, the inner product $\inner{v,w}_\varsigma$ is defined via an integral over $\Omega\setminus C_0$. To more easily distinguish between the weak formulations of $u$ and $u_\epsilon$, we write down \eqref{eq:weakform} specifically for $u_\epsilon$ here:
	\begin{equation}
	\int_{\Omega} \sigma_\epsilon \nabla u_\epsilon \cdot \nabla \overline{v}\,\di x = \inner{f,v|_\Gamma}, \quad \forall v\in H_\diamond^1(\Omega). \label{eq:weakformeps}
	\end{equation} 

        First we give some estimates related to $u_\epsilon$ in terms of the $L^2(\Gamma)$-norm of $f$. Using \eqref{eq:weakformeps}, the ``Cauchy inequality with $\varepsilon$'' with $\alpha>0$ taking the role of $\varepsilon$, and the trace theorem in $\Omega\setminus C$ leads to
	\begin{align*}
	\int_{\Omega} \sigma_\epsilon\abs{\nabla u_\epsilon}^2\,\di x &= \inner{f,u_\epsilon|_\Gamma} \leq \alpha\norm{f}_{L^2(\Gamma)}^2 + \frac{1}{4\alpha}\norm{u_\epsilon|_{\partial\Omega}}^2_{H^{1/2}(\partial\Omega)} \\
	&\leq \alpha\norm{f}_{L^2(\Gamma)}^2 + \frac{K}{4\alpha} \int_{\Omega\setminus C} \varsigma \abs{\nabla u_\epsilon}^2\,\di x.
	\end{align*}
	Choosing $\alpha = K$ and restructuring gives
	\begin{equation*}
	\frac{3}{4}\int_{\Omega\setminus C} \varsigma\abs{\nabla u_\epsilon}^2\,\di x + \epsilon^{-1}\int_{C_\infty} \varsigma\abs{\nabla u_\epsilon}^2\,\di x + \epsilon\int_{C_0} \varsigma\abs{\nabla u_\epsilon}^2\,\di x \leq K \norm{f}_{L^2(\Gamma)}^2.
	\end{equation*}
	As direct consequences,
	\begin{align}
	\norm{\nabla u_\epsilon}_{L^2(C_\infty;\rum{C}^d)} &\leq \epsilon^{1/2} K \norm{f}_{L^2(\Gamma)}, \label{eq:ueps_infty} \\
	\norm{\epsilon^{1/2}\varsigma\nabla u_\epsilon}_{L^2(C_0;\rum{C}^d)} &\leq K \norm{f}_{L^2(\Gamma)}, \label{eq:ueps_0} 
	\end{align}
        where $K>0$ remains a generic constant independent of $f$ and $\epsilon$. These bounds will be used later to finalize the proof.
	
	We proceed now to prove the first part of Theorem~\ref{thm:uconverge}. Let $P = P(\varsigma,C_0,C_\infty)$ and $P_\perp = P_\perp(\varsigma,C_0,C_\infty)$. Since the projections $P$ and $P_\perp$ are self-adjoint and $u \in \mathscr{H}_\diamond$, Cauchy--Schwarz' inequality yields
	\begin{align}
	\norm{u-u_\epsilon}_\varsigma^2 &=  \norm{P(u-u_\epsilon)}_\varsigma^2 + \norm{P_\perp(u-u_\epsilon)}_\varsigma^2 \notag \\  
	&\leq \inner{u-u_\epsilon,P(u-u_\epsilon)}_\varsigma  + \norm{P_\perp u_\epsilon}_\varsigma \norm{u-u_\epsilon}_\varsigma, \label{eq:estu} 
	\end{align} 
        where we have suppressed the notation indicating that the projections are applied to the restriction of $u_\epsilon$ onto $\Omega\setminus C_0$. The rest of the proof will focus on estimating the two terms on the right-hand side of \eqref{eq:estu}, with the aim of finally obtaining the desired convergence result by dividing with $\norm{u-u_\epsilon}_\varsigma$.

	\subsection{Estimating the first term in \eqref{eq:estu}}
	
	By using the fact $\nabla P(u-u_\epsilon)=0$ in $C_\infty^\circ$, it follows that
	\begin{align}
	\inner{u-u_\epsilon,P(u-u_\epsilon)}_\varsigma &= \int_{\Omega\setminus C} \varsigma \nabla u \cdot \nabla \overline{P(u-u_\epsilon)}\,\di x - \int_{\Omega\setminus C_0} \sigma_\epsilon \nabla u_\epsilon \cdot \nabla \overline{P(u-u_\epsilon)}\,\di x \notag \\
	&= \int_{C_0} \epsilon \varsigma \nabla u_\epsilon \cdot \nabla \overline{P(u-u_\epsilon)}\,\di x, 
	\label{eq:firstterm}
	\end{align} 
	where the second step is based on the weak formulations \eqref{eq:weakform} and \eqref{eq:weakformeps} for $u$ and $u_\epsilon$, respectively.
    
	We adopt the notation that $\varsigma\nabla u_\epsilon$ equals $(\varsigma\nabla u_\epsilon)_\textup{int}$ in $C_0^\circ$ and $(\varsigma\nabla u_\epsilon)_\textup{ext}$ in $\Omega\setminus C_0$ to respect the jump-relation $\nu\cdot(\varsigma \nabla u_{\epsilon})_\textup{ext} = \epsilon\nu\cdot(\varsigma \nabla u_{\epsilon})_\textup{int}$ on $\partial C_0$ induced by \eqref{eq:continuum}. Since $u_\epsilon$ satisfies the conductivity equation in $C_0^\circ$, we may rewrite \eqref{eq:firstterm} as
	\begin{align}
	\inner{u-u_\epsilon,P(u-u_\epsilon)}_\varsigma &= \epsilon^{1/2} \inner{\epsilon^{1/2}\nu \cdot (\varsigma \nabla u_\epsilon)_\textup{int}|_{\partial C_0}, P(u-u_\epsilon)|_{\partial C_0}}_{\partial C_0}, \label{eq:firstterm2}
	\end{align}
	where $\inner{\cdot,\cdot}_{\partial C_0} = \inner{\cdot,\cdot}_{H^{-1/2}(\partial C_0),H^{1/2}(\partial C_0)}$ is the dual bracket on $\partial C_0$. Now we can straightforwardly estimate \eqref{eq:firstterm2} using the boundedness of the dual bracket, the trace theorem in $\Omega\setminus C_0$, and the boundedness of $P$,
	\begin{align}
	\inner{u-u_\epsilon,P(u-u_\epsilon)}_\varsigma &\leq \epsilon^{1/2}\norm{\epsilon^{1/2}\nu\cdot (\varsigma\nabla u_{\epsilon})_\textup{int}|_{\partial C_0}}_{H^{-1/2}(\partial C_0)}\norm{P(u-u_\epsilon)|_{\partial C_0}}_{H^{1/2}(\partial C_0)} \notag \\
	&\leq K\epsilon^{1/2}\norm{\epsilon^{1/2}\nu\cdot (\varsigma\nabla u_{\epsilon})_\textup{int}|_{\partial C_0}}_{H^{-1/2}(\partial C_0)}\norm{u-u_\epsilon}_\varsigma. \label{eq:firstterm3}
	\end{align}
	Taking the trace of the normal component is a bounded map from $H_\textup{div}(C_0^\circ)$ to $H^{-1/2}(\partial C_0)$. Hence, the equality $\nabla\cdot(\epsilon^{1/2}\varsigma\nabla u_\epsilon) = \nabla\cdot(\epsilon\varsigma\nabla u_\epsilon) = 0$ in $C_0^\circ$ implies with \eqref{eq:ueps_0} that
	\begin{align}
	\norm{\epsilon^{1/2}\nu\cdot (\varsigma\nabla u_{\epsilon})_\textup{int}|_{\partial C_0}}_{H^{-1/2}(\partial C_0)} &\leq K \left( \norm{\epsilon^{1/2}\varsigma \nabla u_\epsilon}_{L^2(C_0;\rum{C}^d)}^2 + \norm{\nabla\cdot (\epsilon^{1/2}\varsigma\nabla u_\epsilon)}_{L^2(C_0)}^2 \right)^{1/2} \notag \\
	&\leq K\norm{f}_{L^2(\Gamma)}. \label{eq:est2}
	\end{align}
	Combining \eqref{eq:firstterm3} and \eqref{eq:est2} handles the first term of \eqref{eq:estu}. 
	
	\subsection{Estimating the second term in \eqref{eq:estu}} \label{sec:est2}
    The following lemma gives the needed estimate for $\norm{P_\perp u_\epsilon}_\varsigma$ in the second term on the right-hand side of \eqref{eq:estu}. However, the lemma will also be useful in proving some other results in what follows.
	\begin{lemma} \label{lemma:perp}
		Let $C = C_0\cup C_\infty$ satisfy Assumption~\ref{assump}, $\varsigma\in L^\infty_+(\Omega)$, and $P_\perp = P_\perp(\varsigma,C_0,C_\infty)$. If $v \in H_\diamond^1(\Omega\setminus C_0)$ satisfies $\nabla\cdot(\varsigma\nabla v) = 0$ in $\Omega\setminus C$, then there is a constant $K>0$ (independent of $v$) such that
		\begin{equation*}
		\int_{\Omega\setminus C_0} \varsigma \abs{\nabla P_\perp v}^2 \,\di x \leq K\int_{C_\infty} \abs{\nabla v}^2\,\di x.
		\end{equation*}
	\end{lemma}
	\begin{proof}
	  Let $v^{C_\infty}$ denote the piecewise constant function in $H^1(C_\infty^\circ)$ that in each connected component $C_i$ of $C_\infty$ equals the mean value of the corresponding restriction $v|_{C_i}$. We define $w$ as the unique solution in $\mathscr{H}_\diamond$ to the auxiliary Dirichlet problem
		\begin{align*}
		-\nabla \cdot(\varsigma \nabla w) &= 0 \text{ in } \Omega \setminus C, \\
		w &= v \text{ on } \partial(\Omega\setminus C_0), \\
		w &= v^{C_\infty} \text{ in } C_\infty^\circ.
		\end{align*}
		It easily follows that $v-w$ satisfies a similar boundary value problem in $\Omega\setminus C$ with the only nontrivial boundary condition being a Dirichlet condition on $\partial C_\infty$. Using the continuous dependence on the Dirichlet condition for an elliptic equation, the trace theorem in $C_\infty$, and the Poincar\'e inequality in each connected component of $C_\infty$, we obtain the bound
		\begin{align}
		\norm{v - w}^2_{H^1(\Omega\setminus C_0)} &= \norm{v - w}^2_{H^1(\Omega\setminus C)} + \norm{v-v^{C_\infty}}^2_{H^1(C_\infty^\circ)} \notag \\
		&\leq K\left(\norm{(v-v^{C_\infty})|_{\partial C_\infty}}^2_{H^{1/2}(\partial C_\infty)} + \norm{v-v^{C_\infty}}^2_{H^1(C_\infty^\circ)} \right) \notag\\
		&\leq K\norm{v-v^{C_\infty}}^2_{H^1(C_\infty^\circ)} \notag\\
		&\leq K\norm{\nabla v}^2_{L^2(C_\infty;\rum{C}^d)}. \label{eq:appendixest}
		\end{align}
		The projection theorem and \eqref{eq:appendixest} finally yield
		\begin{equation*}
		\norm{P_\perp v}_{\varsigma} = \inf_{\hat{v}\in \mathscr{H}_\diamond}\norm{v - \hat{v}}_\varsigma \leq \norm{v - w}_\varsigma \leq K\norm{\nabla v}_{L^2(C_\infty;\rum{C}^d)}, 
		\end{equation*}
		hence concluding the proof.
	\end{proof}
	Applying Lemma~\ref{lemma:perp} to $u_\epsilon|_{\Omega\setminus C_0}$ and subsequently using \eqref{eq:ueps_infty} gives the bound
	\begin{equation*}
	  \norm{P_\perp u_\epsilon}_{\varsigma} \leq \epsilon^{1/2}K \norm{f}_{L^2(\Gamma)}
	\end{equation*}
        that we were looking for.
        \subsection{Completing the proof}
	Combining the obtained bounds for the two terms on the right-hand side of \eqref{eq:estu} leads to \eqref{eq:convrate} and proves the first part of Theorem~\ref{thm:uconverge}. To prove the second part, we write
	\begin{align*}
	\norm{\Lambda(\sigma_\epsilon) - \Lambda(\sigma)}_{\mathscr{L}(L^2_\diamond(\Gamma))} &= \sup_{\norm{f}_{L^2(\Gamma)}=1} \norm{(u_f^{\sigma_{\epsilon}} - u_f^\sigma)|_\Gamma}_{L^2(\Gamma)} \\
	&\leq \sup_{\norm{f}_{L^2(\Gamma)}=1} K\norm{u_f^{\sigma_{\epsilon}} - u_f^\sigma}_{H^1(\Omega\setminus C_0)} \leq K \epsilon^{1/2},
	\end{align*}
    where we used the trace theorem in $\Omega \setminus C_0$.~\hfill $\square$

	\section{Proofs of Theorem~\ref{thm:monofast} and Theorem~\ref{thm:monoslow}} \label{sec:monomethoddefinite}
	
	We denote by $\Lambda_\alpha(C) := \Lambda(\gamma_{\alpha,C})$, $\alpha > 0$, the ND map for the conductivity
	\begin{equation} \label{eq:truncgamma}
	\gamma_{\alpha,C} := \begin{cases}
	\gamma_0 & \text{in } \Omega\setminus C, \\
	\alpha\gamma_0 & \text{in } C.
	\end{cases}
	\end{equation}
	By Theorem~\ref{thm:uconverge}, we have $\Lambda_\alpha(C) \to \Lambda_0(C)$ in $\mathscr{L}(L^2_\diamond(\Gamma))$ as $\alpha\to 0$ and $\Lambda_\alpha(C)\to \Lambda_\infty(C)$ in $\mathscr{L}(L^2_\diamond(\Gamma))$ as $\alpha\to \infty$ for any $C\in\mathcal{A}$. This will allow us to apply to $\Lambda_\infty(C)$ and $\Lambda_0(C)$ the well known monotonicity properties of non-extreme ND maps via a limiting procedure.
	
    To this end, we recall some monotonicity principles for $\Lambda|_{L^\infty_+(\Omega)}$ as well as a fundamental result on so-called localizing solutions to \eqref{eq:continuum}.
	\begin{lemma}[Monotonicity principles] \label{lemma:mono}
		For $f\in L^2_\diamond(\Gamma)$ and $\varsigma_1, \varsigma_2 \in L^\infty_+(\Omega)$, it holds
		\begin{equation*}
		\int_{\Omega} \frac{\varsigma_2}{\varsigma_1}(\varsigma_1-\varsigma_2) \abs{\nabla u_f^{\varsigma_2}}^2\, \di x \leq \inner{(\Lambda(\varsigma_2) - \Lambda(\varsigma_1))f,f} \leq \int_{\Omega} (\varsigma_1-\varsigma_2)\abs{\nabla u_f^{\varsigma_2}}^2 \, \di x.
		\end{equation*}
	\end{lemma}
	\begin{proof}
		This type of results go back to \cite{Kang1997b,Ikehata1998a}. See \cite[Lemma~3.1]{Harrach13} or \cite[Lemma~2.1]{Harrach10} for a proof of this particular version with $\Gamma = \partial \Omega$, readily modifiable to the case of local ND maps using the weak formulation \eqref{eq:weakform}. See also \cite[Section~4.3]{Harrach13} for remarks on such an extension.
	\end{proof}
	\begin{lemma}[Localized potentials] \label{lemma:locpot}
		Let $U\subset \overline{\Omega}$ be a relatively open connected set that intersects $\Gamma$. Let $B\subset U$ be a nonempty open set and assume $\varsigma\in L^\infty_+(\Omega)$ satisfies the UCP in $U$. Then there are sequences $(f_i)\subset L^2_\diamond(\Gamma)$ and $(u_i)\subset H^1_\diamond(\Omega)$ with $u_i = u_{f_i}^\varsigma$ such that
		\begin{equation}
			\lim_{i\to\infty} \int_B \abs{\nabla u_i}^2\,\di x = \infty \qquad \text{and} \qquad \lim_{i\to\infty}\int_{\Omega\setminus U} \abs{\nabla u_i}^2\,\di x = 0. \label{eq:locpot}
		\end{equation}
	\end{lemma}
	\begin{proof}
		This result is a corollary of \cite[Theorem~2.7]{Gebauer2008b}, which can be seen by using \cite[Theorem~2.7]{Gebauer2008b} to localize potentials inside an open ball $\tilde{B}$ satisfying $\overline{\tilde{B}}\subset B$.
	\end{proof}
	We note that one also obtains localized potentials in the case of extreme inclusions, with the same sequence of current densities $(f_i)$ as for the background conductivity, as long as the extreme inclusions are located in the part of the domain where the gradient tends to zero. In the following result, it may be worth recalling the definition of the extension operator into the insulating parts from Corollary~\ref{coro:extended}.
	\begin{lemma} \label{lemma:locpot2}
		Let the sets $U$ and $B$ be given as in Lemma~\ref{lemma:locpot}. For $\varsigma\in L^\infty_+(\Omega)$ and $(f_i)\subset L^2_\diamond(\Gamma)$, suppose that $u_i = u_{f_i}^\varsigma$ satisfies
		\begin{equation}
			\lim_{i\to\infty} \int_B \abs{\nabla u_i}^2\,\di x = \infty \qquad \text{and} \qquad \lim_{i\to\infty}\int_{\Omega\setminus U} \abs{\nabla u_i}^2\,\di x = 0. \label{eq:locpotagain}
		\end{equation}
		If $C = C_0\cup C_\infty$ satisfies Assumption~\ref{assump} with $C \subset \Omega\setminus \overline{U}$ and moreover $\hat{u}_i = E(\varsigma,C_0) u_{f_i}^\sigma$ with $\sigma = \sigma(\varsigma,C_0,C_\infty)$, then it also holds
		\begin{equation}
		\lim_{i\to\infty} \int_B \abs{\nabla \hat{u}_i}^2\,\di x = \infty \qquad \text{and} \qquad \lim_{i\to\infty}\int_{\Omega\setminus U} \abs{\nabla \hat{u}_i}^2\,\di x = 0. \label{eq:locpot2}
		\end{equation}
	\end{lemma}
	\begin{proof}
		Set $\hat{v}_i = u_{f_i}^\sigma$ and use Proposition~\ref{prop:altcharacterization} to write
		\begin{equation*}
			\hat{v}_i = Pv_i,
		\end{equation*}		
		where $v_i = u_i|_{\Omega\setminus C_0} - w_i$ and $w_i \in H^1_\diamond(\Omega \setminus C_0)$ solves the auxiliary problem in Proposition~\ref{prop:altcharacterization} with $u^\varsigma_f$ replaced by $u_i$. In the following we denote $(\varsigma\nabla u_i)_\textup{ext} := \varsigma\nabla u_i|_{\Omega\setminus C_0}$.

              By the continuous dependence of $w_i$ on the Neumann condition and a bound similar to \eqref{eq:est2} involving $H_{\rm div}(\Omega \setminus (V\cup C_0))$ instead of $H_{\rm div}(C_0)$, we have
		\begin{align} \label{eq:wiest}
			\norm{\nabla w_i}_{L^2(\Omega\setminus C_0;\rum{C}^d)} &\leq \norm{w_i}_{H^1(\Omega\setminus C_0)} \leq K\norm{\nu\cdot (\varsigma\nabla u_i)_\textup{ext}|_{\partial C_0}}_{H^{-1/2}(\partial C_0)} \notag\\
			&\leq K\norm{\nabla u_i}_{L^2(\Omega\setminus (V\cup C_0);\rum{C}^d)} \leq K\norm{\nabla u_i}_{L^2(\Omega\setminus (U\cup C_0);\rum{C}^d)},
		\end{align}
		where $V$ is chosen such that $C_0 \subset \Omega\setminus\overline{V}\subset \Omega\setminus \overline{U}$ and $\partial(\Omega\setminus \overline{V})$ is Lipschitz continuous. From \eqref{eq:locpotagain} and \eqref{eq:wiest}, and because $K>0$ is independent of $f_i$, we have that $\nabla w_i$ converges to a zero-vector in $L^2(\Omega\setminus C_0;\rum{C}^d)$. Therefore the limiting behavior of $\nabla v_i$ coincides with the limiting behavior of $\nabla u_i|_{\Omega\setminus C_0}$. By applying \eqref{eq:locpotagain} again we deduce that
		\begin{equation} \label{eq:vhatest1}
			\lim_{i\to\infty} \int_B \abs{\nabla v_i}^2\,\di x = \infty \qquad \text{and} \qquad \lim_{i\to\infty}\int_{\Omega\setminus (U\cup C_0)} \abs{\nabla v_i}^2\,\di x = 0.
		\end{equation}

                We write $\hat{v}_i = v_i - P_\perp v_i$ and focus on the latter term. Lemma~\ref{lemma:perp} gives the bound
		\begin{equation} \label{eq:vhatest2}
			\norm{\nabla P_\perp v_i}_{L^2(\Omega\setminus C_0;\rum{C}^d)} \leq K\norm{P_\perp v_i}_\varsigma \leq K\norm{\nabla v_i}_{L^2(C_\infty;\rum{C}^d)}.
		\end{equation}
		Since $C_\infty \subset \Omega\setminus (\overline{U}\cup C_0)$, the combination of \eqref{eq:vhatest1} and \eqref{eq:vhatest2} gives
		\begin{equation*} 
		\lim_{i\to\infty} \int_B \abs{\nabla \hat{v}_i}^2\,\di x = \infty \qquad \text{and} \qquad \lim_{i\to\infty}\int_{\Omega\setminus (U\cup C_0)} \abs{\nabla \hat{v}_i}^2\,\di x = 0.
		\end{equation*}
		As $\hat{u}_i|_{\Omega\setminus C_0} = \hat{v}_i$, what remains is to investigate $\hat{u}_i|_{C_0}$.
		
		Since $\hat{u}_i|_{\Omega\setminus(\overline{U}\cup C_0)}$ converges to zero in the quotient space $H^1(\Omega\setminus(\overline{U}\cup C_0))/\rum{C}$, it follows from an obvious extension of the trace theorem that $\hat{u}_i|_{\partial C_0}$ tends to zero in $H^{1/2}(\partial C_0)/\rum{C}$. Note that $\nabla \hat{u}_i = \nabla \tilde{u}_i$ in $C_0^\circ$, where $\tilde{u}_i$ is the unique solution in $H^1(C_0^\circ)/\rum{C}$ to $\nabla\cdot(\varsigma\nabla \tilde{u}_i) = 0$ in $C_0^\circ$ with its Dirichlet trace on $\partial C_0$ being the element of $H^{1/2}(\partial C_0)/\rum{C}$ defined by $\hat{u}_i|_{\partial C_0} \in H^{1/2}(\partial C_0)$. Due to the continuous dependence of $\tilde{u}_i$ on its Dirichlet trace in the respective quotient topologies, we have 
		\begin{equation*}
			\lim_{i\to\infty}\int_{C_0} \abs{\nabla \hat{u}_i}^2\,\di x = \lim_{i\to\infty}\int_{C_0} \abs{\nabla \tilde{u}_i}^2\,\di x = 0,
		\end{equation*} 
		thereby concluding the proof.
	\end{proof}
	
	\begin{remark}
	  In the proofs of Theorem~\ref{thm:monofast}, Theorem~\ref{thm:monoslow}, and Theorem~\ref{thm:general}, we first employ the monotonicity principles from Lemma~\ref{lemma:mono} for truncated conductivities, and then show ``$\Rightarrow$'' of the theorem statements by passing to the limit with the help of Theorem~\ref{thm:uconverge}. It is tempting to also try something similar in the opposite direction ``$\Leftarrow$'' by proving the contrapositive of the theorem statements. Indeed, the $\epsilon$-truncated conductivities define an analytic (see e.g.~\cite[Theorem~A.2]{Garde_2019c}) family of compact self-adjoint operators, converging in the operator norm, such that for each small enough $\epsilon>0$ one of the listed inequalities is contradicted. Hence, for the associated operators, the bottom of the spectrum is strictly negative for $\epsilon >0$. Going to the limit $\epsilon\to 0$ implies that the bottom of the spectrum, say $\lambda$, is less than or equal to zero for the limiting compact self-adjoint operator as well. If one could guarantee that $\lambda$ is an eigenvalue of the limiting operator, then one could also easily establish that actually $\lambda < 0$, which would prove the direction ``$\Leftarrow$'' of the theorems. For Theorem~\ref{thm:monofast} and Theorem~\ref{thm:monoslow} this could be done by comparing with the analogous operator with a smaller $\beta$, and for Theorem~\ref{thm:general} by slightly magnifying the test inclusion $C$.
          However, it turns out that it is indeed possible to construct families of compact self-adjoint operators, depending analytically on a positive parameter and converging in the operator norm as this parameter tends to zero, such that the bottom of the spectrum converges from negative to zero, but without zero being an eigenvalue of the limiting operator. To circumvent this problem, we will prove the direction ``$\Leftarrow$'' of the aforementioned theorems by an alternative technique relying on the tighter monotonicity estimates established in Appendix~\ref{sec:appendixA}.
	\end{remark}
	
	We now state a straightforward corollary that will be useful specifically in proving the direction ``$\Leftarrow$'' in Theorem~\ref{thm:monofast} and Theorem~\ref{thm:monoslow}. 
	\begin{corollary} \label{coro:mononew}
		Let $f\in L^2_\diamond(\Gamma)$, $\varsigma\in L_+^\infty(\Omega)$, and $\hat{\gamma}_0 = \sigma(\gamma_0,D,\emptyset)$. By denoting $\hat{u} = E(\gamma_0,D)u_f^{\hat{\gamma}_0}$ and $u = u_f^{\gamma_0}$, it holds
		\begin{align*}
		\inner{(\Lambda_0(D) - \Lambda(\varsigma))f,f} &\leq \int_{\Omega\setminus D} (\varsigma-\gamma_0)\abs{\nabla \hat{u}}^2 \, \di x + \int_{D} \varsigma \abs{\nabla \hat{u}}^2\, \di x, \\
		\inner{(\Lambda(\varsigma) - \Lambda_\infty(D))f,f} &\leq \int_{\Omega} \frac{\gamma_0}{\varsigma}(\gamma_0-\varsigma)\abs{\nabla u}^2\,\di x + K\int_{D} \abs{\nabla u}^2\,\di x,
		\end{align*}
		where $K>0$ is independent of $f$.
	\end{corollary}
	\begin{proof}
Applying Lemma~\ref{lemma:mono} to the truncated conductivity $\gamma_{\alpha,D}$ from \eqref{eq:truncgamma} gives
		\begin{equation} \label{eq:firstmonoineqtemp}
		\inner{(\Lambda(\gamma_{\alpha,D}) - \Lambda(\sigma))f,f} \leq \int_{\Omega} (\sigma-\gamma_{\alpha,D})\abs{\nabla u_f^{\gamma_{\alpha,D}}}^2 \, \di x.
		\end{equation}
		Since \eqref{eq:firstmonoineqtemp} holds for all $\alpha>0$,  taking the limit $\alpha\to 0$ proves the first inequality by virtue of Theorem~\ref{thm:uconverge} and Corollary~\ref{coro:extended}.
		
		To prove the second inequality, consider
		\begin{equation} \label{eq:secondmonoineqtmp}
			\inner{(\Lambda(\varsigma) - \Lambda_\infty(D))f,f} = \inner{(\Lambda(\varsigma) - \Lambda(\gamma_0))f,f} + \inner{(\Lambda(\gamma_0) - \Lambda_\infty(D))f,f}.
		\end{equation}
		The proof is completed by a direct application of Lemma~\ref{lemma:mono} to the first term on the right-hand side of \eqref{eq:secondmonoineqtmp} and of Lemma~\ref{lemma:monoappendix}(ii) to the second term on the right-hand side of \eqref{eq:secondmonoineqtmp}.
	\end{proof}
	
	\subsection{Proof of Theorem~\ref{thm:monofast}}
	Let $\alpha_0 > 0$ be fixed and assume $0<\beta \leq \frac{\alpha_0}{1+\alpha_0} \inf(\gamma_0)$. At the end of the proof, we will let $\alpha_0\to\infty$, and hence allow $\beta$ to take values up to and including $\inf(\gamma_0)$.
	
	We start by considering the first ``if and only if'' statement. 
	
	\subsubsection*{Proof of ``$\Rightarrow$''} Assume $B\subset D$ and consider values of $\alpha$ in the range $0<\alpha \leq 1 - \frac{\alpha_0}{1+\alpha_0}$. Applying Lemma~\ref{lemma:mono} and \eqref{eq:DNfrechet} leads to
	\begin{align*}
	\inner{(\Lambda_\alpha(D) - \Lambda(\gamma_0) + \beta \DLambda(\gamma_0;\chi_B))f,f} &\geq \int_D \left[ (1-\alpha)\gamma_0 - \beta\chi_B \right] \abs{\nabla u_f^{\gamma_0}}^2\,\di x  \\
	&\geq \frac{\alpha_0}{1+\alpha_0} \int_B \left[\gamma_0 - \inf(\gamma_0)\right]\abs{\nabla u_f^{\gamma_0}}^2\,\di x \geq 0
	\end{align*}
	for all $f\in L^2_\diamond(\Gamma)$. Hence $B\subset D \Rightarrow \Lambda_\alpha(D)\geq \Lambda(\gamma_0)-\beta \DLambda(\gamma_0;\chi_B)$ for any $0<\alpha \leq 1 - \frac{\alpha_0}{1+\alpha_0}$, i.e.\ letting $\alpha\to 0$ gives $B\subset D \Rightarrow \Lambda_0(D)\geq \Lambda(\gamma_0)-\beta \DLambda(\gamma_0;\chi_B)$.
	
	\subsubsection*{Proof of ``$\Leftarrow$''} Assume $B\not\subset D$ so that $B\setminus D$ contains an open ball $\tilde{B}$, in which we may concentrate the localized potentials via a suitable set $U$. To this end, let $U\subset\overline{\Omega}$ be a relatively open connected set that intersects $\Gamma$, contains $\tilde{B}$, and satisfies $\overline{U}\cap D = \emptyset$. Annotating $\hat{\gamma}_0 = \sigma(\gamma_0,D,\emptyset)$, we now pick $(f_i)\subset L^2_\diamond(\Gamma)$ via Lemma~\ref{lemma:locpot} and Lemma~\ref{lemma:locpot2} such that the corresponding potentials $\hat{u}_i = E(\gamma_0,D)u_{f_i}^{\hat{\gamma}_0}$ and $u_i = u_{f_i}^{\gamma_0}$ satisfy
	\begin{equation*}
	\lim_{i\to\infty}\int_D \abs{\nabla \hat{u}_i}^2\,\di x \leq \lim_{i\to\infty}\int_{\Omega\setminus U} \abs{\nabla \hat{u}_i}^2\,\di x = 0, \qquad \lim_{i\to\infty}\int_B\abs{\nabla u_i}^2\,\di x \geq \lim_{i\to\infty}\int_{\tilde{B}}\abs{\nabla u_i}^2\,\di x = \infty. 
	\end{equation*}
	By Corollary~\ref{coro:mononew} and \eqref{eq:DNfrechet}, we have
	\begin{equation*}
	\inner{(\Lambda_0(D) - \Lambda(\gamma_0) + \beta \DLambda(\gamma_0;\chi_B))f_i,f_i} \leq \sup_D(\gamma_0)\int_D \abs{\nabla \hat{u}_i}^2\,\di x - \beta\int_B \abs{\nabla u_i}^2\,\di x \to -\infty,
	\end{equation*}
	and thus $B\not\subset D \Rightarrow \Lambda_0(D) \not\geq \Lambda(\gamma_0) - \beta \DLambda(\gamma_0;\chi_B)$.

        \bigskip
        \bigskip
        
	Now consider the second ``if and only if'' statement, that is, ``$B\subset D \Leftrightarrow \Lambda(\gamma_0) + \beta \DLambda(\gamma_0;\chi_B) \geq \Lambda_\infty(D)$''. 
	
	\subsubsection*{Proof of ``$\Rightarrow$''} Assume $B\subset D$ and $\alpha \geq 1+\alpha_0$. Lemma~\ref{lemma:mono} and \eqref{eq:DNfrechet} imply
	\begin{align*}
	\inner{(\Lambda(\gamma_0) + \beta\DLambda(\gamma_0;\chi_B) - \Lambda_\alpha(D))f,f} & \geq \int_D \left[(1-\alpha^{-1})\gamma_0 - \beta\chi_B\right]\abs{\nabla u_f^{\gamma_0}}^2\,\di x \\
	&\geq \frac{\alpha_0}{1+\alpha_0} \int_B \left[\gamma_0 - \inf(\gamma_0)\right] \abs{\nabla u_f^{\gamma_0}}^2\,\di x \geq 0
	\end{align*}
	for all $f\in L^2_\diamond(\Gamma)$. 
        Taking the limit $\alpha\to\infty$ leads to $B\subset D \Rightarrow \Lambda(\gamma_0) + \beta \DLambda(\gamma_0;\chi_B) \geq \Lambda_\infty(D)$.
	
	\subsubsection*{Proof of ``$\Leftarrow$''} Now assume $B\not\subset D$ and let us localize $u_i$ in the same manner as in the insulating case, i.e.
	\begin{equation*}
	\lim_{i\to\infty}\int_D \abs{\nabla u_i}^2\,\di x \leq \lim_{i\to\infty}\int_{\Omega\setminus U} \abs{\nabla u_i}^2\,\di x = 0, \qquad \lim_{i\to\infty}\int_B\abs{\nabla u_i}^2\,\di x \geq \lim_{i\to\infty}\int_{\tilde{B}}\abs{\nabla u_i}^2\,\di x = \infty. 
	\end{equation*}
	Corollary~\ref{coro:mononew} and \eqref{eq:DNfrechet} thus yield
	\begin{equation*}
	\inner{(\Lambda(\gamma_0) + \beta\DLambda(\gamma_0;\chi_B) - \Lambda_\infty(D))f_i,f_i} \leq K\int_{D} \abs{\nabla u_i}^2\,\di x - \beta \int_B\abs{\nabla u_i}^2\,\di x \to -\infty,
	\end{equation*}
	and we conclude that $B\not\subset D \Rightarrow \Lambda(\gamma_0) + \beta\DLambda(\gamma_0;\chi_B)\not\geq \Lambda_\infty(D)$.

        \bigskip
        
	Since all above results hold for any $\alpha_0 > 0$, we may actually choose $0 < \beta \leq \inf(\gamma_0)$ as in the theorem statement by letting $\alpha_0\to\infty$.~\hfill $\square$

	\subsection{Proof of Theorem~\ref{thm:monoslow}}
	
	Let us consider the first ``if and only if'' statement. Hence, we assume $0 < \beta < \inf(\gamma_0)$ and choose $0 < \alpha_0 < 1$ to be small enough so that $\alpha_0 \gamma_0 \leq \gamma_0-\beta \chi_B$. 
	
	\subsubsection*{Proof of ``$\Rightarrow$''} Assume $B \subset D$, which implies $\gamma_{\alpha,D} \leq \gamma_0-\beta\chi_B$ for $0 < \alpha \leq \alpha_0$. Using the monotonicity principles in Lemma \ref{lemma:mono}, we directly have $\Lambda_\alpha (D) \geq \Lambda (\gamma_0 - \beta\chi_B)$ for all $0<\alpha \leq \alpha_0$. Letting $\alpha\to 0 $ gives $B\subset D \Rightarrow \Lambda_0(D) \geq \Lambda(\gamma_0-\beta\chi_B)$.
	
	\subsubsection*{Proof of ``$\Leftarrow$''} Now assume $B\not\subset D$.  We pick $U$ and $\tilde{B}$ as in the proof of Theorem~\ref{thm:monofast} and localize $(\hat{u}_i)$ in precisely the same manner as for the insulating case in the proof of Theorem~\ref{thm:monofast}:
	\begin{equation*}
	\lim_{i\to\infty}\int_{\Omega\setminus U} \abs{\nabla \hat{u}_i}^2\,\di x = 0 \qquad\text{and}\qquad \lim_{i\to\infty}\int_{\tilde{B}}\abs{\nabla \hat{u}_i}^2\,\di x = \infty. 
	\end{equation*}
	Hence, by Corollary~\ref{coro:mononew},
	\begin{align*}
	\inner{(\Lambda_0(D)-\Lambda(\gamma_0-\beta\chi_B))f_i,f_i} &\leq -\beta\int_{B\setminus D} \abs{\nabla \hat{u}_i}^2\,\di x + \int_D (\gamma_0 - \beta\chi_B)\abs{\nabla \hat{u}_i}^2\,\di x \\
	&\leq -\beta\int_{\tilde{B}} \abs{\nabla \hat{u}_i}^2\,\di x + \sup_D(\gamma_0-\beta\chi_B)\int_{\Omega\setminus U} \abs{\nabla \hat{u}_i}^2\,\di x.
	\end{align*}
	Letting $i\to\infty$ implies $B\not\subset D\Rightarrow \Lambda_0(D) \not\geq\Lambda(\gamma_0-\beta\chi_B$).

        \bigskip
        \bigskip
        
	Let us then consider the second ``if and only if'' statement. We therefore assume $\beta >0$ and choose $\alpha_0 > 0$ large enough so that $\alpha_0\gamma_0 \geq \gamma_0+\beta \chi_B$. 
	
	\subsubsection*{Proof of ``$\Rightarrow$''} Assume $B \subset D$, which implies $\gamma_0+\beta\chi_B \leq \gamma_{\alpha,D}$ for $\alpha\geq \alpha_0$. Due to the monotonicity principles in Lemma~\ref{lemma:mono}, it follows that $\Lambda (\gamma_0+\beta\chi_B) \geq \Lambda_{\alpha} (D)$ for all $\alpha \geq \alpha_0$, so letting $\alpha\to \infty$ gives $B\subset D \Rightarrow \Lambda(\gamma_0+\beta\chi_B) \geq \Lambda_\infty(D)$.
	
	\subsubsection*{Proof of ``$\Leftarrow$''} Now assume $B\not\subset D$ and localize $(u_i)$ as in the proof of Theorem~\ref{thm:monofast}, i.e.
	\begin{equation*}
	\lim_{i\to\infty}\int_{\Omega\setminus U} \abs{\nabla u_i}^2\,\di x = 0 \qquad\text{and}\qquad \lim_{i\to\infty}\int_{\tilde{B}}\abs{\nabla u_i}^2\,\di x = \infty. 
	\end{equation*}
	Once again we apply Corollary~\ref{coro:mononew} to write
	\begin{align*}
	\inner{(\Lambda(\gamma_0+\beta\chi_B)-\Lambda_{\infty}(D))f_i,f_i} &\leq -\beta\int_B \frac{\gamma_0}{\gamma_0+\beta}\abs{\nabla u_i}^2\,\di x + K\int_D\abs{\nabla u_i}^2\,\di x \\
	&\leq -\beta\inf_B\left(\frac{\gamma_0}{\gamma_0+\beta}\right)\int_{\tilde{B}} \abs{\nabla u_i}^2\,\di x + K\int_{\Omega\setminus U} \abs{\nabla u_i}^2\,\di x.
	\end{align*}
	Letting $i\to\infty$ implies $B\not\subset D\Rightarrow \Lambda(\gamma_0+\beta\chi_B)\not\geq \Lambda_{\infty}(D)$, thereby concluding the proof.~\hfill $\square$
	
	\section{Proof of Theorem~\ref{thm:general}} \label{sec:monomethodindefinite}
	
	\subsubsection*{Proof of ``$\Rightarrow$''} 
	
	Define the $\epsilon$-truncation of $\gamma$, with $\epsilon > 0$, as
	\begin{equation*}
		\gamma_\epsilon := \begin{cases}
			\epsilon\gamma_0 & \text{in } D_0, \\
			\epsilon^{-1}\gamma_0 & \text{in } D_\infty, \\
			\gamma_- & \text{in } D_\textup{F}^-, \\
			\gamma_+ & \text{in } D_\textup{F}^+, \\
			\gamma_0 & \text{in } \Omega\setminus D.
		\end{cases}
	\end{equation*}
	In particular, we have $\Lambda(\gamma_{\epsilon}) \to \Lambda(\gamma)$ in $\mathscr{L}(L^2_\diamond(\Gamma))$ as $\epsilon\to 0$ by Theorem~\ref{thm:uconverge}. In the following we assume $0 < \epsilon_0 < 1$ is small enough so that $\epsilon_0\gamma_0 \leq \gamma$ in $D_\textup{F}^-$ and $\epsilon_0^{-1}\gamma_0 \geq \gamma$ in $D_\textup{F}^+$.
	
	Assume $D\subseteq C$ and let $0<\epsilon\leq \epsilon_0$. Then it holds (cf.~\eqref{eq:truncgamma})
	\begin{equation*}
	\gamma_{\epsilon,C} \leq \gamma_{\epsilon} \leq \gamma_{\epsilon^{-1},C}.
	\end{equation*}
	In particular, by virtue of the monotonicity principles in Lemma~\ref{lemma:mono}, we have $\Lambda_{\epsilon}(C) \geq \Lambda(\gamma_{\epsilon}) \geq \Lambda_{\epsilon^{-1}}(C)$ for all $0<\epsilon\leq \epsilon_0$. Letting $\epsilon\to 0$ gives $D\subseteq C \Rightarrow \Lambda_0(C) \geq \Lambda(\gamma) \geq \Lambda_\infty(C)$.
	
	\subsubsection*{Proof of ``$\Leftarrow$''} Now assume $D\not\subseteq C$. Due to Definition \ref{def:indefinite}, there exist an open ball $B\subset D\setminus C$ and a relatively open connected set $U\subset\overline{\Omega}$ such that $U$ intersects $\Gamma$,  $B\subset U$, $\overline{U}\cap C = \emptyset$, and one of the following four options holds:
	\begin{alignat*}{3}
		\text{(a)}&\text{: } \overline{U}\cap (D\setminus D_\textup{F}^+) &= \emptyset, \qquad \text{(b)}&\text{: } \overline{U}\cap (D\setminus D_\textup{F}^-)\, &= \emptyset,\\
		\text{(c)}&\text{: } \overline{U}\cap (D\setminus D_\infty)\, &= \emptyset, \qquad \text{(d)}&\text{: } \overline{U}\cap (D\setminus D_0) &= \emptyset. 
	\end{alignat*}
	We will consider these four cases separately. However, first we need to introduce some auxiliary conductivity coefficients and potentials. 
	
	Let us define $\gamma^\textup{F} \in L^\infty_+(\Omega)$ as
	\begin{equation*}
		\gamma^\textup{F} := \begin{cases}
			\gamma_- & \text{in } D_\textup{F}^-, \\
			\gamma_+ & \text{in } D_\textup{F}^+, \\
			\gamma_0 & \text{in } \Omega\setminus (D_\textup{F}^- \cup D_\textup{F}^+)
		\end{cases}
	\end{equation*}
	and introduce six (extreme) conductivity coefficients:
	\begin{alignat*}{2}
		\gamma_1 &= \sigma(\gamma_0,C,\emptyset), \qquad\quad\qquad & \\
		\gamma_2 &= \sigma(\gamma_0,D_0,\emptyset), &\gamma_{2}^\textup{F} &= \sigma(\gamma^\textup{F},D_0,\emptyset), \\
		\gamma_3 &= \sigma(\gamma_0,\emptyset,D_\infty),   &\gamma_{3}^\textup{F} &= \sigma(\gamma^\textup{F},\emptyset,D_\infty), \\
		\gamma_4 &= \sigma(\gamma_0,D_0,D_\infty).
	\end{alignat*}
	Furthermore, for $(f_i), (g_i) \subset L^2_\diamond(\Gamma)$, we define eight auxiliary potential sequences via
	\begin{alignat*}{2}
		u_{0,i} &= u_{f_i}^{\gamma_0}, \qquad\qquad\qquad\qquad &\tilde{u}_{0,i} &= u_{g_i}^{\gamma_0}, \\
		u_{1,i} &= E(\gamma_0,C) u_{f_i}^{\gamma_1}, &\tilde{u}_{1,i} &= E(\gamma_0,C) u_{g_i}^{\gamma_1}, \\
		u_{2,i} &= E(\gamma_0,D_0)u_{f_i}^{\gamma_2}, &\tilde{u}_{2,i} &= E(\gamma^\textup{F},D_0)u_{g_i}^{\gamma_2^\textup{F}}, \\
		u_{3,i} &= u_{f_i}^{\gamma_3}, & \tilde{u}_{3,i} &= u_{g_i}^{\gamma_3^\textup{F}}, \\
		u_{4,i} &= u_{f_i}^{\gamma_4}, &\tilde{u}_{4,i} &= u_{g_i}^{\gamma^\textup{F}}.
	\end{alignat*}

        In cases (a) and (b) we pick $(f_i)$, via Lemma~\ref{lemma:locpot} and Lemma~\ref{lemma:locpot2}, such that for $j\in\{0,1,2,3,4\}$ it holds
	\begin{equation} \label{eq:caseABlimit}
		\lim_{i\to\infty}\int_{\Omega\setminus U} \abs{\nabla u_{j,i}}^2\,\di x = 0 \qquad \text{and} \qquad \lim_{i\to\infty}\int_{B} \abs{\nabla u_{j,i}}^2\,\di x = \infty. 
	\end{equation}
	Recall that $\gamma^\textup{F}$ equals $\gamma_0$ away from $D_\textup{F}^\pm$ and thus satisfies the UCP there. According to~\cite[Lemma~3.7]{Harrach13},  $(\tilde{u}_{0,i})$ and $(\tilde{u}_{4,i})$ can thus be localized in $B$ along the set $U$ by resorting to the same sequence of current densities $(g_i)$ in cases (c) and (d). In consequence, we may pick $(g_i)$, via Lemma~\ref{lemma:locpot} and Lemma~\ref{lemma:locpot2}, such that
	\begin{equation} \label{eq:caseCDlimit}
	\lim_{i\to\infty}\int_{\Omega\setminus U} \abs{\nabla \tilde{u}_{j,i}}^2\,\di x = 0 \qquad \text{and} \qquad \lim_{i\to\infty}\int_{B} \abs{\nabla \tilde{u}_{j,i}}^2\,\di x = \infty 
	\end{equation}
	holds for $j\in \{0,1,2,4\}$ in case (c) and for $j\in\{0,1,3,4\}$ in case (d).
	
	We will go in turns through the four cases (a)--(d) and contradict one of the two inequalities in Theorem~\ref{thm:general}. More precisely, cases (b) and (d) contradict the first inequality, whereas cases (a) and (c) contradict the second inequality.
	
	\subsection*{Case (a)} 
	
	As noted above, the aim is to contradict the second inequality in Theorem~\ref{thm:general}. To this end, we write
	\begin{align}
		\inner{(\Lambda(\gamma)-\Lambda_\infty(C))f_i,f_i} &= \inner{(\Lambda(\gamma)-\Lambda(\gamma_4))f_i,f_i} + \inner{(\Lambda(\gamma_4)-\Lambda(\gamma_2))f_i,f_i} \notag\\
		&\hphantom{={}} + \inner{(\Lambda(\gamma_2)-\Lambda(\gamma_0))f_i,f_i} + \inner{(\Lambda(\gamma_0)-\Lambda_\infty(C))f_i,f_i}. \label{eq:casea}
	\end{align}
	 We now apply (i), (ii), (iii), and (ii) of Lemma~\ref{lemma:monoappendix}, respectively, to the four terms on the right-hand side of \eqref{eq:casea} to arrive at
	\begin{align}
	\inner{(\Lambda(\gamma)-\Lambda(\gamma_4))f_i,f_i} &\leq \int_{D_\textup{F}^-}\frac{\gamma_0}{\gamma_-}(\gamma_0 - \gamma_-)\abs{\nabla u_{4,i}}^2\,\di x + \int_{D_\textup{F}^+}\frac{\gamma_0}{\gamma_+}(\gamma_0 - \gamma_+)\abs{\nabla u_{4,i}}^2\,\di x, \label{eq:caseaest1} \\
	\inner{(\Lambda(\gamma_4)-\Lambda(\gamma_2))f_i,f_i} &\leq -\inf_{D_\infty}(\gamma_0)\int_{D_\infty} \abs{\nabla u_{2,i}}^2\, \di x, \label{eq:caseaest2}\\
	\inner{(\Lambda(\gamma_2)-\Lambda(\gamma_0))f_i,f_i} &\leq \sup_{D_0}(\gamma_0)\int_{D_0} \abs{\nabla u_{2,i}}^2\,\di x, \label{eq:caseaest3} \\
	\inner{(\Lambda(\gamma_0)-\Lambda_\infty(C))f_i,f_i} &\leq K\int_{C_\infty} \abs{\nabla u_{0,i}}^2\,\di x. \label{eq:caseaest4}
	\end{align}
	Using \eqref{eq:caseABlimit}, it follows that \eqref{eq:caseaest2}--\eqref{eq:caseaest4} tend to zero as $i\to\infty$. Since $\gamma_+ \geq \gamma_0$ in $D_\textup{F}^+ \supset B$ and $\gamma_- \leq \gamma_0$ in $D_\textup{F}^- \subset \Omega\setminus\overline{U}$, we obtain from \eqref{eq:caseaest1} that
	\begin{align} \label{eq:caseaest1new}
		\inner{(\Lambda(\gamma)-\Lambda(\gamma_4))f_i,f_i} &\leq \sup_{D_\textup{F}^-}\left( \frac{\gamma_0}{\gamma_-}(\gamma_0 - \gamma_-) \right)\int_{\Omega\setminus U}\abs{\nabla u_{4,i}}^2\,\di x \notag\\
		&\hphantom{\leq{}}- \inf_{D_\textup{F}^+}\left(\frac{\gamma_0}{\gamma_+}(\gamma_+ - \gamma_0)\right)\int_{B}\abs{\nabla u_{4,i}}^2\,\di x. 
	\end{align}
	According to Definition~\ref{def:indefinite}, it holds that $\inf_{D_\textup{F}^+}(\gamma_+-\gamma_0) > 0$. Since $\gamma_0\in L^\infty_+(\Omega)$ and $\gamma_\pm\in L^\infty_+(D_\textup{F}^\pm)$, \eqref{eq:caseaest1new} and \eqref{eq:caseABlimit} yield
	\begin{equation*}
		\lim_{i\to\infty}\inner{(\Lambda(\gamma)-\Lambda_\infty(C))f_i,f_i} = \lim_{i\to\infty}\inner{(\Lambda(\gamma)-\Lambda(\gamma_4))f_i,f_i}=-\infty.
	\end{equation*}
	
	\subsection*{Case (b)}
	
	In this case, the goal is to contradict the first inequality in Theorem~\ref{thm:general}. For this purpose, we write
	\begin{align}
		\inner{(\Lambda_0(C)-\Lambda(\gamma))f_i,f_i} &= \inner{(\Lambda_0(C)-\Lambda(\gamma_0))f_i,f_i} + \inner{(\Lambda(\gamma_0)-\Lambda(\gamma_3))f_i,f_i} \notag\\
		&\hphantom{={}}+ \inner{(\Lambda(\gamma_3)-\Lambda(\gamma_4))f_i,f_i} + \inner{(\Lambda(\gamma_4)-\Lambda(\gamma))f_i,f_i}. \label{eq:caseb}
	\end{align}
        We now apply (iii), (ii), (iii), and (i) of Lemma~\ref{lemma:monoappendix}, respectively, to the four terms on the right-hand side of \eqref{eq:caseb} to conclude
	\begin{align}
		\inner{(\Lambda_0(C)-\Lambda(\gamma_0))f_i,f_i} &\leq \sup_C(\gamma_0)\int_{C} \abs{\nabla u_{1,i}}^2\,\di x, \label{eq:casebest1} \\
		\inner{(\Lambda(\gamma_0)-\Lambda(\gamma_3))f_i,f_i} &\leq K\int_{D_\infty} \abs{\nabla u_{0,i}}^2\,\di x, \label{eq:casebest2}\\
		\inner{(\Lambda(\gamma_3)-\Lambda(\gamma_4))f_i,f_i} &\leq -\inf_{D_0}(\gamma_0)\int_{D_0} \abs{\nabla u_{3,i}}^2\,\di x, \label{eq:casebest3} \\
		\inner{(\Lambda(\gamma_4)-\Lambda(\gamma))f_i,f_i} &\leq \int_{D_\textup{F}^-} (\gamma_- - \gamma_0) \abs{\nabla u_{4,i}}^2\,\di x + \int_{D_\textup{F}^+} (\gamma_+ - \gamma_0) \abs{\nabla u_{4,i}}^2\,\di x. \label{eq:casebest4}
	\end{align}
	Due to \eqref{eq:caseABlimit}, it is clear that \eqref{eq:casebest1}--\eqref{eq:casebest3} tend to zero as $i\to\infty$. Since $\gamma_- \leq \gamma_0$ in $D_\textup{F}^-  \supset B$ and $\gamma_+ \geq \gamma_0$ in $D_\textup{F}^+  \subset \Omega\setminus\overline{U}$, the estimate \eqref{eq:casebest4} leads to
	\begin{equation} \label{eq:casebest4new}
	\inner{(\Lambda(\gamma_4)-\Lambda(\gamma))f_i,f_i} \leq \sup_{D_\textup{F}^-}(\gamma_- - \gamma_0)\int_{B} \abs{\nabla u_{4,i}}^2\,\di x + \sup_{D_\textup{F}^+}(\gamma_+ - \gamma_0)\int_{\Omega\setminus U}  \abs{\nabla u_{4,i}}^2\,\di x.
	\end{equation}
	According to Definition~\ref{def:indefinite}, it holds that $\sup_{D_\textup{F}^-}(\gamma_- - \gamma_0) < 0$. As a consequence, \eqref{eq:casebest4new} and \eqref{eq:caseABlimit} guarantee that 
	\begin{equation*}
		\lim_{i\to\infty}\inner{(\Lambda_0(C)-\Lambda(\gamma))f_i,f_i} = \lim_{i\to\infty}\inner{(\Lambda(\gamma_4)-\Lambda(\gamma))f_i,f_i}=-\infty.
	\end{equation*}
	
	\subsection*{Case (c)} 
	This time around we will contradict the second inequality in Theorem~\ref{thm:general}. We start by writing
	\begin{align}
	\inner{(\Lambda(\gamma)-\Lambda_\infty(C))g_i,g_i} &= \inner{(\Lambda(\gamma)-\Lambda(\gamma_2^\textup{F}))g_i,g_i} + \inner{(\Lambda(\gamma_2^\textup{F})-\Lambda(\gamma^\textup{F}))g_i,g_i} \notag\\
	&\hphantom{={}} + \inner{(\Lambda(\gamma^\textup{F})-\Lambda(\gamma_0))g_i,g_i} + \inner{(\Lambda(\gamma_0)-\Lambda_\infty(C))g_i,g_i}. \label{eq:casec}
	\end{align}
        Resorting once again to Lemma~\ref{lemma:monoappendix}, we apply its subresults (ii), (iii), (i), and (ii), respectively, to the four terms on the right-hand side of \eqref{eq:casec} to demonstrate that
	\begin{align}
	\inner{(\Lambda(\gamma)-\Lambda(\gamma_2^\textup{F}))g_i,g_i} &\leq -\inf_{D_\infty}(\gamma_0)\int_{D_\infty} \abs{\nabla \tilde{u}_{2,i}}^2\, \di x, \label{eq:casecest1}\\
	\inner{(\Lambda(\gamma_2^\textup{F})-\Lambda(\gamma^\textup{F}))g_i,g_i} &\leq \sup_{D_0}(\gamma_0)\int_{D_0} \abs{\nabla \tilde{u}_{2,i}}^2\,\di x, \label{eq:casecest2} \\
	\inner{(\Lambda(\gamma^\textup{F})-\Lambda(\gamma_0))g_i,g_i} & \leq \sup_{D_\textup{F}^- \cup D_\textup{F}^+}(\gamma_0-\gamma^\textup{F})\int_{D_\textup{F}^-\cup D_\textup{F}^+} \abs{\nabla \tilde{u}_{4,i}}^2\,\di x, \label{eq:casecest3} \\
	\inner{(\Lambda(\gamma_0)-\Lambda_\infty(C))g_i,g_i} &\leq K\int_{C_\infty} \abs{\nabla \tilde{u}_{0,i}}^2\,\di x. \label{eq:casecest4}
	\end{align}
	Since $D_\textup{F}^\pm \subset \Omega\setminus\overline{U}$, it follows directly from \eqref{eq:caseCDlimit} that \eqref{eq:casecest2}--\eqref{eq:casecest4} tend to zero and \eqref{eq:casecest1} goes to minus infinity as $i\to\infty$, i.e.
	\begin{equation*}
	\lim_{i\to\infty}\inner{(\Lambda(\gamma)-\Lambda_\infty(C))g_i,g_i} = \lim_{i\to\infty}\inner{(\Lambda(\gamma)-\Lambda(\gamma_2^\textup{F}))g_i,g_i}=-\infty.
	\end{equation*}
	
	\subsection*{Case (d)}
	
	In this final case, we will contradict the first inequality in Theorem~\ref{thm:general}. To achieve this objective, we first write
	\begin{align}
	\inner{(\Lambda_0(C)-\Lambda(\gamma))g_i,g_i} &= \inner{(\Lambda_0(C)-\Lambda(\gamma_0))g_i,g_i} + \inner{(\Lambda(\gamma_0)-\Lambda(\gamma^\textup{F}))g_i,g_i}  \notag\\
	&\hphantom{={}}+ \inner{(\Lambda(\gamma^\textup{F})-\Lambda(\gamma_3^\textup{F}))g_i,g_i} + \inner{(\Lambda(\gamma_3^\textup{F})-\Lambda(\gamma))g_i,g_i}. \label{eq:cased}
	\end{align}
        Analogously to the previous three cases, we apply (iii), (i), (ii), and (iii) of Lemma~\ref{lemma:monoappendix}, in this order, to the four terms on the right-hand side of \eqref{eq:cased} to reach the estimates
	\begin{align}
	\inner{(\Lambda_0(C)-\Lambda(\gamma_0))g_i,g_i} &\leq \sup_C(\gamma_0)\int_{C} \abs{\nabla \tilde{u}_{1,i}}^2\,\di x, \label{eq:casedest1} \\
	\inner{(\Lambda(\gamma_0)-\Lambda(\gamma^\textup{F}))g_i,g_i} & \leq \sup_{D_\textup{F}^- \cup D_\textup{F}^+}(\gamma^\textup{F} - \gamma_0)\int_{D_\textup{F}^-\cup D_\textup{F}^+} \abs{\nabla \tilde{u}_{0,i}}^2\,\di x, \label{eq:casedest2} \\
	\inner{(\Lambda(\gamma^\textup{F})-\Lambda(\gamma_3^\textup{F}))g_i,g_i} &\leq K\int_{D_\infty} \abs{\nabla \tilde{u}_{4,i}}^2\,\di x, \label{eq:casedest3}\\
	\inner{(\Lambda(\gamma_3^\textup{F})-\Lambda(\gamma))g_i,g_i} &\leq -\inf_{D_0}(\gamma_0)\int_{D_0} \abs{\nabla \tilde{u}_{3,i}}^2\,\di x. \label{eq:casedest4}
	\end{align}
	Since $D_\textup{F}^\pm \subset \Omega\setminus\overline{U}$, it follows directly from \eqref{eq:caseCDlimit} that \eqref{eq:casedest1}--\eqref{eq:casedest3} tend to zero and \eqref{eq:casedest4} converges to minus infinity as $i\to\infty$, leading to the sought-for conclusion
	\begin{equation*}
	\lim_{i\to\infty}\inner{(\Lambda_0(C)-\Lambda(\gamma))g_i,g_i} = \lim_{i\to\infty}\inner{(\Lambda(\gamma_3^\textup{F})-\Lambda(\gamma))g_i,g_i}=-\infty.
	\end{equation*}
	%

        
	The results for the cases (a)--(d) complete the proof of ``$\Leftarrow$'', that is, $D\not\subseteq C$ indeed implies either $\Lambda_0(C) \not\geq \Lambda(\gamma)$ or $\Lambda(\gamma) \not\geq \Lambda_\infty(C)$.
	
	To conclude the proof of Theorem~\ref{thm:general}, and simultaneously the whole paper, the established ``if and only if'' statement guarantees $D\subseteq C$ for all sets $C$ belonging to
	\begin{equation*}
		\mathcal{M}:=\{ C\in\mathcal{A} \mid \Lambda_0(C) \geq \Lambda(\gamma) \geq \Lambda_\infty(C) \}.
	\end{equation*}
	Since $D$ itself is also a member of $\mathcal{M}$, this proves $D = \cap \mathcal{M}$.~\hfill $\square$
	
	\subsection*{Acknowledgments} 
	
	This work was supported by the Academy of Finland (decision 312124) and the Aalto Science Institute (AScI). The authors thank Antti Hannukainen (Aalto University) and Arne Jensen (Aalborg University) for useful discussions.
	
	\appendix
	
	\section{Some monotonicity estimates} \label{sec:appendixA}
	
	In this appendix, we give some generic monotonicity estimates that are extensively utilized in the proof Theorem~\ref{thm:general}. For more information on the notation and definitions, we refer to Section~\ref{sec:continuum}. See also Corollary~\ref{coro:extended} for the definition of the extension operator into the perfectly insulating parts.
	
	In the following $f\in L^2_\diamond(\Gamma)$ is arbitrary but fixed. We denote $u_{\varsigma,C_0,C_\infty} = E(\varsigma,C_0) u_f^\sigma$ with $\sigma = \sigma(\varsigma,C_0,C_\infty)$, and we analogously set $\Lambda_{\varsigma,C_0,C_\infty} = \Lambda(\sigma)$. To motivate such notation, take note that we will consider several cases where $C_0$, $C_\infty$, or both are empty sets and we also choose $\varsigma$ in several ways.
	\begin{lemma} \label{lemma:monoappendix}
		Let $C = C_0\cup C_\infty$ satisfy Assumption~\ref{assump} and $\varsigma,\varsigma_1,\varsigma_2\in L^\infty_+(\Omega)$. 
		\begin{enumerate}[(i)]
			\item Different background conductivity, but with the same extreme inclusions:
			\begin{equation*}
				\int_{\Omega\setminus C} \frac{\varsigma_2}{\varsigma_1}(\varsigma_1-\varsigma_2)\abs{\nabla u_{\varsigma_2,C_0,C_\infty}}^2\,\di x \leq \inner{(\Lambda_{\varsigma_2,C_0,C_\infty}-\Lambda_{\varsigma_1,C_0,C_\infty})f,f} \leq \int_{\Omega\setminus C} (\varsigma_1 - \varsigma_2)\abs{\nabla u_{\varsigma_2,C_0,C_\infty}}^2\,\di x.
			\end{equation*}
			\item Same background conductivity and perfectly insulating inclusions, but with and without perfectly conducting inclusions:
			\begin{equation*}
			 \int_{C_\infty} \varsigma\abs{\nabla u_{\varsigma,C_0,\emptyset}}^2\,\di x \leq \inner{(\Lambda_{\varsigma,C_0,\emptyset}-\Lambda_{\varsigma,C_0,C_\infty})f,f} \leq K\int_{C_\infty} \abs{\nabla u_{\varsigma,C_0,\emptyset}}^2\,\di x,
			\end{equation*}
			where $K>0$ is independent of $f$.
			\item Same background conductivity and perfectly conducting inclusions, but with and without perfectly insulating inclusions:
			\begin{equation*}
			  \int_{C_0} \varsigma\abs{\nabla u_{\varsigma,\emptyset,C_\infty}}^2\,\di x \leq \inner{(\Lambda_{\varsigma,C_0,C_\infty}-\Lambda_{\varsigma,\emptyset,C_\infty})f,f} \leq \int_{C_0} \varsigma \abs{\nabla u_{\varsigma,C_0,C_\infty}}^2\,\di x. 
			\end{equation*}
		\end{enumerate}
	\end{lemma}
	\subsubsection*{Proof of (i)} 
	
	Abbreviate $\sigma_j = \sigma(\varsigma_j,C_0,C_\infty)$, $u_j = u_f^{\sigma_j}$, and $\Lambda_j = \Lambda(\sigma_j)$ for $j\in\{1,2\}$. 
	
	Using the weak forms \eqref{eq:weakform} for $u_1$ and $u_2$ as well as the self-adjointness of $\Lambda_1$, we have
	\begin{align}
		\inner{(\Lambda_2-\Lambda_1)f,f} &= \int_{\Omega\setminus C} \left(\varsigma_1\abs{\nabla u_1}^2 + \varsigma_2\abs{\nabla u_2}^2 - \varsigma_2\nabla u_2\cdot\nabla \overline{u_1} - \varsigma_2\nabla u_1\cdot\nabla \overline{u_2} \right) \di x  \notag\\
		&= \int_{\Omega\setminus C} \varsigma_1 \Big| \nabla u_1 - \frac{\varsigma_2}{\varsigma_1}\nabla u_2\Big|^2 \,\di x + \int_{\Omega\setminus C} \frac{\varsigma_2}{\varsigma_1}(\varsigma_1-\varsigma_2)\abs{\nabla u_2}^2\,\di x.   \label{eq:appest1}
	\end{align}
	Estimating the first term on the right-hand side of \eqref{eq:appest1} from below by zero gives the left-hand inequality in (i).

	Using this time around the self-adjointness of $\Lambda_2$, we obtain
	\begin{align}
	\inner{(\Lambda_2-\Lambda_1)f,f} &= \int_{\Omega\setminus C} \left( (\varsigma_1-\varsigma_2)\abs{\nabla u_2}^2 - \varsigma_1\abs{\nabla u_1}^2 - \varsigma_1\abs{\nabla u_2}^2 + \varsigma_1\nabla u_1\cdot\nabla\overline{u_2} + \varsigma_1\nabla u_2\cdot\nabla \overline{u_1} \right) \di x  \notag\\
	&= \int_{\Omega\setminus C} (\varsigma_1-\varsigma_2)\abs{\nabla u_2}^2 \,\di x - \int_{\Omega\setminus C} \varsigma_1 \abs{\nabla (u_1-u_2)}^2\,\di x.   \label{eq:appest2}
	\end{align}
	Estimating the second term on the right-hand side of \eqref{eq:appest2} from above by zero gives the right-hand inequality in (i).

	\subsubsection*{Proof of (ii)}
	
	Abbreviate $\sigma_1 = \sigma(\varsigma,C_0,C_\infty)$ and $\sigma_2 = \sigma(\varsigma,C_0,\emptyset)$, and in the same way as above, let $u_j = u_f^{\sigma_j}$ and $\Lambda_j = \Lambda(\sigma_j)$ for $j\in\{1,2\}$. We also abbreviate $P = P(\varsigma,C_0,C_\infty)$ and $P_\perp = P_\perp(\varsigma,C_0,C_\infty)$.
	
	Let us then introduce an $\epsilon$-truncation of $\sigma_1$ in $C_\infty$,
	\begin{equation*}
		\sigma_{\epsilon} = \begin{cases}
			\varsigma &\text{in } \Omega\setminus C, \\
			0 & \text{in } C_0, \\
			\epsilon^{-1}\varsigma &\text{in } C_\infty.
		\end{cases}
	\end{equation*}
	Using the lower bound of Lemma~\ref{lemma:mono}, with $\Omega$ replaced by $\Omega\setminus C_0$, we get
	\begin{equation*}
		(1-\epsilon)\int_{C_\infty} \varsigma \abs{\nabla u_2}^2\, \di x \leq \inner{(\Lambda_2-\Lambda(\sigma_\epsilon))f,f}.
	\end{equation*}
	Due to Theorem~\ref{thm:uconverge}, the left-hand inequality in (ii) follows by taking the limit $\epsilon\to 0$.
	
	To obtain the right-hand inequality in (ii), note first that $u_1 = Pu_2$ due to Proposition~\ref{prop:altcharacterization}. In particular, it holds that $u_2 = u_1 + P_\perp u_2$, and thus the weak formulation for $u_2$ gives
	\begin{align*}
		\inner{(\Lambda_2-\Lambda_1)f,f} &= \inner{f,(\Lambda_2-\Lambda_1)f} = \inner{f,u_2|_\Gamma} - \inner{f,u_1|_\Gamma} = \inner{f,(P_\perp u_2)|_\Gamma} \\
		&= \int_{\Omega\setminus C_0} \varsigma \nabla u_2\cdot \nabla \overline{P_\perp u_2}\,\di x = \int_{\Omega\setminus C_0} \varsigma \abs{\nabla P_\perp u_2}^2\,\di x.
	\end{align*}
	The right-hand inequality of (ii) is now a simple consequence Lemma~\ref{lemma:perp}.
	
	\subsubsection*{Proof of (iii)}
	
	Abbreviate $\sigma_1 = \sigma(\varsigma,C_0,C_\infty)$ and $\sigma_2 = \sigma(\varsigma,\emptyset,C_\infty)$, let $u_1 = E(\varsigma,C_0) u_f^{\sigma_1}$ and $u_2 = u_f^{\sigma_2}$, and introduce the corresponding ND maps $\Lambda_j = \Lambda(\sigma_j)$ for $j\in\{1,2\}$.
	
	To prove the left-hand inequality of (iii), we introduce a truncation of $\sigma_1$ in $C_0$,
	\begin{equation*}
	\sigma_{\epsilon} = \begin{cases}
	\varsigma &\text{in } \Omega\setminus C, \\
	\epsilon\varsigma &\text{in } C_0, \\
	\infty &\text{in } C_\infty.
	\end{cases}
	\end{equation*}
        When there are no perfectly insulating inclusions, the upper bound of (i) gives
	\begin{equation}
		\inner{(\Lambda_2-\Lambda(\sigma_\epsilon))f,f} \leq (\epsilon-1)\int_{C_0} \varsigma\abs{\nabla u_2}^2 \,\di x.
	\end{equation}
	 Due to the convergence result of Theorem~\ref{thm:uconverge}, taking the limit $\epsilon\to 0$ gives the left-hand inequality of (iii). (In fact, it is enough to consider the first term on the right-hand side of \eqref{eq:estu}.)
	
	Using the weak formulations for $u_1$ and $u_2$, along with the self-adjointness of $\Lambda_1$, we have
	\begin{equation*}
	\inner{\Lambda_1 f,f} = \int_{\Omega\setminus C} \varsigma\abs{\nabla u_1}^2\,\di x = \int_{\Omega\setminus C_\infty} \varsigma \nabla u_2\cdot\nabla\overline{u_1}\,\di x = \int_{\Omega\setminus C_\infty} \varsigma \nabla u_1\cdot\nabla\overline{u_2}\,\di x.
	\end{equation*}
	It follows that
	\begin{align}
	\int_{\Omega\setminus C_\infty} \varsigma\abs{\nabla(u_1-u_2)}^2 \,\di x &= \int_{\Omega\setminus C_\infty} \left(\varsigma\abs{\nabla u_1}^2 + \varsigma\abs{\nabla u_2}^2 - \varsigma\nabla u_1\cdot\nabla \overline{u_2} - \varsigma\nabla u_2\cdot\nabla \overline{u_1} \right) \,\di x \notag \\
	&= \inner{(\Lambda_2-\Lambda_1)f,f} + \int_{C_0}\varsigma\abs{\nabla u_1}^2\,\di x. \label{eq:appestfinal}
	\end{align}
	Since the left-hand side of \eqref{eq:appestfinal} is nonnegative, we have proven the right-hand inequality in (iii) and thereby completed the whole proof.~\hfill $\square$
	
	\section{Ambiguity in the perfectly insulating parts} \label{sec:appendixB}
	
	In this appendix, we give a simple example on the limiting behavior of the electric potential when the conductivity coefficient decays to zero in some parts of the domain. The example shows that the restriction of the limit potential to the insulating part of the domain may be different for different conductivity sequences even if they converge to the same limit conductivity.

	
	Let $\Omega$ be the unit disk and let $\Gamma = \partial \Omega$ be the unit circle. We make use of the polar coordinates $(r,\theta)\in (0,1)\times(-\pi,\pi]$ and consider the following two types of radially symmetric conductivity coefficients for $\epsilon > 0$:
	\begin{equation*}
		\sigma_\epsilon(r,\theta) := \begin{cases}
			\epsilon &\textup{for } r\in(0,\frac{1}{2}), \\
			1 &\textup{for } r\in(\frac{1}{2},1),
		\end{cases} \qquad \qquad
		\hat{\sigma}_\epsilon(r,\theta) := \begin{cases}
		\epsilon^2 &\textup{for } r\in(0,\frac{1}{4}), \\
		\epsilon &\textup{for } r\in(\frac{1}{4},\frac{1}{2}), \\
		1 &\textup{for } r\in(\frac{1}{2},1).
		\end{cases}
	\end{equation*}
	If $f = \e^{\I(\cdot)}$ is the applied current density as a function of the polar angle on $\Gamma$, then the corresponding electric potentials can be found via separation of variables:
	\begin{align*}
	u_\epsilon(r,\theta) &= \frac{\e^{\I\theta}}{3+5\epsilon}\begin{cases}
	8r &\textup{for } r\in(0,\frac{1}{2}), \\
	4(1+\epsilon)r + (1-\epsilon)r^{-1} &\textup{for } r\in(\frac{1}{2},1),
	\end{cases} \\[2mm]
	\hat{u}_\epsilon(r,\theta) &= \frac{\e^{\I\theta}}{15+24\epsilon+25\epsilon^2}\begin{cases}
	64r &\textup{for } r\in(0,\frac{1}{4}), \\
	32(1+\epsilon)r + 2(1-\epsilon)r^{-1} &\textup{for } r\in(\frac{1}{4},\frac{1}{2}), \\
	4(5+6\epsilon+5\epsilon^2)r + 5(1-\epsilon^2)r^{-1} &\textup{for } r\in(\frac{1}{2},1).
	\end{cases}
	\end{align*}
	As $\epsilon\to 0$, the resulting limit potentials are obviously 
	\begin{equation*}
	u_0(r,\theta) = \frac{\e^{\I\theta}}{3}\begin{cases}
	8r &\textup{for } r\in(0,\frac{1}{2}), \\
	4r + r^{-1} &\textup{for } r\in(\frac{1}{2},1),
	\end{cases} \qquad
	\hat{u}_0(r,\theta) = \frac{\e^{\I\theta}}{15}\begin{cases}
	64r &\textup{for } r\in(0,\frac{1}{4}), \\
	32r + 2r^{-1} &\textup{for } r\in(\frac{1}{4},\frac{1}{2}), \\
	20r + 5r^{-1} &\textup{for } r\in(\frac{1}{2},1).
	\end{cases}
	\end{equation*}
	It is easy to verify that $u_0$ is the extension introduced in Corollary~\ref{coro:extended}. As expected, $u_0$ and $\hat{u}_0$ agree in the conducting part of the domain, i.e.~for $r\in(\frac{1}{2},1)$. However, they differ in the insulating disk, i.e.~for $r\in (0,\frac{1}{2})$, even though the limits of the corresponding conductivity coefficients $\sigma_\epsilon$ and $\hat{\sigma}_\epsilon$ are identical.
	
	\bibliographystyle{plain}
	\bibliography{minbib}

\end{document}